\documentclass{amsart}
\usepackage{amsmath}
\usepackage{amssymb}
\usepackage{amsthm}
\usepackage{dsfont}
\usepackage{lineno}
\usepackage{subfigure}
\usepackage{graphicx}
\usepackage{multirow}
\usepackage{color}
\usepackage{xspace}
\usepackage{pdfsync}
\usepackage{hyperref} 
\usepackage{algorithmicx}
\usepackage{algpseudocode}
\usepackage{algorithm}
\usepackage{mathtools}
\usepackage{enumitem}
\usepackage{epstopdf}
\usepackage{relsize}
\usepackage{wrapfig}
\graphicspath{{Figures/}}
\usepackage{xcolor}

\newcommand{\bq}{\begin{equation}}
\newcommand{\eq}{\end{equation}}
\newcommand{\R}{\mathbb{R}}

\newcommand{\abs}[1]{\left\vert#1\right\vert}
\newcommand{\norm}[1]{\left\vert#1\right\vert}

\newcommand{\G}{\mathcal{G}}
\newcommand{\bO}{\mathcal{O}}

\newcommand{\Dt}{\mathcal{D}}

\newcommand{\Af}{\mathcal{A}}
\newcommand{\Sf}{\mathcal{S}}

\newcommand{\Nf}{\mathcal{N}}
\newcommand{\MA}{Monge-Amp\`ere\xspace}

\newcommand{\Omb}[0]{\overline{\Omega}}

\newcommand{\x}[0]{\times}

\newcommand{\pOm}[0]{\partial \Omega}
\newcommand{\ra}[0]{\rightarrow}
\newcommand{\bracs}[1]{\left(#1\right)}

\newcommand{\utt}[0]{u_{\theta\theta}}
\newcommand{\maop}[0]{\det(D^2u(x))}

\algnewcommand{\LineComment}[1]{\State \(\triangleright\) #1}

\newtheorem{theorem}{Theorem}

\theoremstyle{lemma}
\newtheorem{lemma}[theorem]{Lemma}

\newtheorem{corollary}[theorem]{Corollary}

\newtheorem{definition}[theorem]{Definition}

\theoremstyle{remark}

\makeatletter
\newcommand\appendix@section[1]{%
\refstepcounter{section}%
\orig@section*{Appendix \@Alph\c@section: #1}%
}
\let\orig@section\section
\g@addto@macro\appendix{\let\section\appendix@section}
\makeatother

\begin{document}

\title[A quadrature method for Monge-Amp\'ere]{A convergent quadrature based method for the Monge-Amp\`ere equation}

\author{Jake Brusca}
\address{Department of Mathematical Sciences, New Jersey Institute of Technology, University Heights, Newark, NJ 07102}
\email{jb327@njit.edu}
\author{Brittany Froese Hamfeldt}
\address{Department of Mathematical Sciences, New Jersey Institute of Technology, University Heights, Newark, NJ 07102}
\email{bdfroese@njit.edu}

\thanks{The authors were partially supported by NSF DMS-1751996.}

\begin{abstract}
We introduce an integral representation of the Monge-Amp\`ere equation, which leads to a new finite difference method based upon numerical quadrature. The resulting scheme is monotone and fits immediately into existing convergence proofs for the Monge-Amp\`ere equation with either Dirichlet or optimal transport boundary conditions.  The use of higher-order quadrature schemes allows for substantial reduction in the component of the error that depends on the angular resolution of the finite difference stencil.  This, in turn, allows for significant improvements in both stencil width and formal truncation error.  {The resulting schemes can achieve a formal accuracy that is arbitrarily close to $\bO(h^2)$, which is the optimal consistency order for monotone approximations of second order operators.} We present {three} different implementations of this method.  The first {two} exploit the spectral accuracy of the trapezoid rule on uniform angular discretizations to allow for computation on a nearest-neighbors finite difference stencil over a large range of grid refinements.  The {third} uses higher-order quadrature to produce superlinear convergence while simultaneously utilizing narrower stencils than other monotone methods. Computational results are presented in two dimensions for problems of various regularity.
\end{abstract}

\date{\today}    
\maketitle
\section{Introduction}\label{sec:intro}
In this article we introduce an integral representation of the \MA equation
\bq\label{eq:MA}
\begin{cases}
-\det(D^2u(x)) + f(x)=0, & x \in \Omega\\
u \text{ is convex}
\end{cases}
\eq
where $\Omega\subset\R^n$ is convex and the right-hand side $f$ is non-negative. 
  This allows us to produce new monotone approximation schemes via quadrature.  Because these schemes are monotone, they fit within several recently developed numerical convergence frameworks~\cite{BenamouDuval_MABVP2,Bonnet_OTBC,FO_MATheory,Hamfeldt_Gauss,Hamfeldt_OTBC,ObermanEigenvalues}.  Moreover, these new schemes offer significant advantages over existing monotone methods in terms of both accuracy and efficiency.
  
  Recent years have seen a growing interest in \MA type equations in the context of a diverse range of problems including design of optical systems~\cite{Romijn_OTDesignSummary}, geophysics~\cite{EF_FWI}, mesh generation~\cite{BuddMeshGen}, medical image processing~\cite{Haker}, meteorology~\cite{CullenOcean1991}, and data science~\cite{Peyre_data}.  This has encouraged the design of many new methods for the \MA equation including~\cite{BFO_MA,BrennerNeilanMA2D,DGnum2006,FengNeilan,Prins_BVP2}.
  
  The development of numerical methods that are guaranteed to converge to the correct solution, particularly in the absence of classical solutions, has proven to be more challenging.  An early method~\cite{olikerprussner88} used a geometric interpretation of weak solutions to design a convergent, but computationally expensive, method for the 2D \MA equation.  Recently, convergence frameworks have been established for the \MA equation with either Dirichlet boundary conditions~\cite{FO_MATheory,Hamfeldt_Gauss,Nochetto_MAConverge,ObermanEigenvalues}:
  \bq\label{eq:dirichlet}
  u(x) = g(x), \quad x \in \partial\Omega
  \eq
  or the second type boundary condition arising in optimal transport~\cite{BenamouDuval_MABVP2,Bonnet_OTBC,Hamfeldt_OTBC}:
  \bq\label{eq:otbc}
  \nabla u(\Omega) \subset \tilde{\Omega}.
  \eq
  
  These convergence proofs can be viewed as extensions of the powerful Barles and Souganidis convergence framework~\cite{BSnum}, which is valid for weak (viscosity) solutions of fully nonlinear partial differential equations.  Critically, they are only valid for approximation schemes that are monotone. Construction of monotone schemes for degenerate elliptic PDE operators is not trivial: in fact, given any fixed finite difference stencil, it is possible to find linear elliptic operators for which no consistent, monotone approximation is possible on the given stencil~\cite{Kocan,MotzkinWasow}. Circumventing this challenge requires the use of finite difference stencils that are allowed to grow wider as the grid is refined. {Monotone schemes are inherently limited in their accuracy: a monotone approximation of a second-order operator can achieve at most second-order ($\bO(h^2)$) truncation error~\cite[Theorem~4]{ObermanEP}.}  
	
	Several monotone finite difference schemes are now available for the \MA equation~\cite{benamou2014monotone,Bonnet_OTBC,FroeseMeshfreeEigs,FO_MATheory,mirebeau2015MA,ObermanEigenvalues}.
  Because of the wide-stencil nature of these methods, the methods are computationally expensive and typically have low (sub-linear) accuracy.  {There are limited techniques available that are capable of achieving the optimal $\bO(h^2)$ truncation error~\cite{benamou2014monotone,Bonnet_OTBC}, but these schemes are valid for the \MA equation only in two-dimensions and with problem data that guarantees uniform ellipticity of the PDE.}  These challenges are magnified in three dimensions, where even evaluating the finite difference approximations (without attempting to solve the resulting nonlinear system) can be prohibitively expensive~\cite{HL_ThreeDimensions}.

In this article, we propose to express the \MA operator in terms of a Gaussian integral.  This allows us to utilize higher-order quadrature schemes in order to simultaneously achieve improved consistency error {(of $\bO(h^{2-\epsilon})$ for any $\epsilon>0$)} and {more compact wide} finite difference stencils.  The schemes are nevertheless monotone, and fit neatly within the existing proofs of convergence of numerical methods to the weak (viscosity) solution of the \MA equation.  We describe {three} different implementations of this approach in two dimensions and validate the performance using a range of standard benchmark problems for the Dirichlet problem.  This new formulation of the \MA equation holds particular promise for the development of computationally practical methods in three dimensions, as it provides a dimension-reduction as compared with a typical variational formulation of the 3D \MA equation.  It also extends naturally to more general \MA type equations in optimal transport, including equations that are posed on the sphere~\cite{HT_OTonSphereNumerics}.

\section{Background}\label{sec:background}
\subsection{Elliptic equations}

The \MA equation is an example of a degenerate elliptic partial differential equation, which take{s} the general form
\bq\label{eq:elliptic}
F(x,u(x),\nabla u(x), D^2u(x)) = 0, \quad x \in \bar{\Omega}.
\eq

\begin{definition}[Degenerate Elliptic]
Let $\Omega\subset\R^n$ and denote by $\Sf^n$ the set of symmetric $n\x n$ matrices.  The operator $F:\Omb\x\R\x\R^n\x\Sf^n \to \R$ is said to be \emph{degenerate elliptic} if 
\begin{equation*}
    F(x,u,p,X) \leq F(x,v,p,Y)
\end{equation*}
whenever $u \leq v$ and $ X \succeq Y$.
\end{definition}
We note that the operator is defined on the closure of $\Omega$, and takes on the value of the relevant boundary conditions at $\pOm$. For the Dirichlet problem, which is the setting implemented in this article, the PDE operator at the boundary is defined as
\bq\label{eq:dirichletOp}
F(x,u(x),\nabla u(x),D^2u(x)) = u(x) - g(x), \quad x \in \partial\Omega.
\eq

The \MA equation~\eqref{eq:MA} does not immediately satisfy this definition of an elliptic equation; in fact, it holds only on the restricted class of convex functions.  Going hand-in-hand with this difficulty is the fact that the solution of the \MA equation is not expected to be unique; the additional constraint that $u$ is convex is needed in order to select a unique solution.  A common remedy to these challenges is to define a globally elliptic extension of the \MA equation that automatically enforces solution convexity~\cite{Hamfeldt_Gauss}.  This can be accomplished by considering the convexified \MA operator
\bq\label{eq:MAconvex}
F(x,u(x),\nabla u(x), D^2u(x)) = -{\det}^+(D^2u(x)) + f(x), \quad x \in \Omega.
\eq
Here the modified determinant ${\det}^+$ should agree with the usual determinant when operating on the Hessian of a convex function and should return a negative value otherwise.  The particular choice utilized in this article is
\bq\label{eq:detPlus}
{\det}^+(M) = \begin{cases} \det(M), & M \succeq 0\\ \lambda_1(M), & \text{otherwise.}  \end{cases}
\eq
where $\lambda_1(M) \leq \ldots \leq \lambda_n(M)$ are the eigenvalues of the symmetric matrix $M$.

In general, degenerate elliptic equations need not have classical solutions, and some notion of weak solution is required.  {The Aleksandrov solution provides a geometric interpretation in terms of the subgradient measure, which allows for very general right-hand sides, including measures that do not have an associated density~\cite{Gutierrez}. Though slightly less general, the viscosity solution has proved to be particularly useful for this class of equations~\cite{CIL}, and forms the foundation for most of the recently developed numerical convergence proofs for the \MA equation.} The idea of the viscosity solution is to use a maximum principle argument to pass derivatives onto smooth test functions that lie above or below the semi-continuous envelopes of the candidate weak solution.

\begin{definition}[Semi-continuous envelopes]
Let $u:\Omega \ra \R$ be a bounded function. Then for $x\in\Omb$, the \emph{upper and lower semi-continuous envelopes} are defined, respectively, as
\begin{equation*}
    u^*(x) = \limsup_{y\ra x} u(y), \quad u_*(x) = \liminf_{y\ra x} u(y).
\end{equation*}
\end{definition}

\begin{definition}[Viscosity subsolutions (supersolutions)]
A bounded upper (lower) semi-continuous function $u$ is a \emph{viscosity subsolution (supersolution)} of~\eqref{eq:elliptic} if for every $\phi \in C^2(\Omb)$, that whenever $u-\phi$ has a local maximum (minimum) at $x\in\Omb$, then
\begin{equation*}
    F_*^{(*)}(x,u(x),\nabla \phi(x),D^2\phi(x)) \leq (\geq) 0.
\end{equation*}
\end{definition}
\begin{definition}[Viscosity Solution]
A bounded function $u:\Omb \ra \R$ is a \emph{viscosity solution} of~\eqref{eq:elliptic} if $u^*(x)$ is a viscosity subsolution and $u_*(x)$ is a viscosity supersolution.
\end{definition}

An important characteristic of many elliptic operators, which immediately yields solution uniqueness, is the comparison principle.
\begin{definition}[Comparison Principle]
The operator~\eqref{eq:elliptic} satisfies a \emph{strong comparison principle} if whenever $u$ is an upper semi-continuous subsolution and $v$ a lower semi-continuous supersolution, then $u \leq v$ on $\Omb$
\end{definition}

We remark that the Dirichlet problem for the \MA equation~\eqref{eq:MA},\eqref{eq:dirichlet} does satisfy a comparison principle under reasonable assumptions on the data. However, this is no longer true when the right-hand side depends on the solution gradient or when the second type boundary condition~\eqref{eq:otbc} is considered~\cite{Hamfeldt_Gauss,Hamfeldt_OTBC}. 

\subsection{Convergence framework}
A fruitful technique for numerically solving fully nonlinear elliptic equations involves finite difference schemes of the form
\begin{equation}\label{eq:scheme}
    F^{h}(x,u(x),u(x) - u(\cdot)) = 0
\end{equation}
defined on a finite set of discretization points $\G\subset\Omega$.
Many key results on the convergence of finite difference methods to the viscosity solution of a degenerate elliptic PDE are based upon a set of criterion developed by Barles and Souganidis~\cite{BSnum}.

\begin{definition}[Consistency]\label{def:consist}
The scheme~\eqref{eq:scheme} is \emph{consistent} with~\eqref{eq:elliptic} if, for any test function $\phi \in C^{2,1}(\Omb)$ and $x \in \Omb$, we have
\begin{align}
    \limsup_{h\ra 0^+,y\ra x, \xi \ra 0}F^h(y,\phi(y)+\xi,\phi(y)-\phi(\cdot)) \leq F^*(x,\phi(x),\nabla\phi(x),D^2\phi(x))
    \\
    \liminf_{h\ra 0^+,y\ra x, \xi \ra 0}F^h(y,\phi(y)+\xi,\phi(y)-\phi(\cdot)) \geq F_*(x,\phi(x),\nabla\phi(x),D^2\phi(x)).
\end{align}
\end{definition}

To a consistent scheme we can also assign a local truncation error.
\begin{definition}[Truncation error]\label{def:truncation}
The \emph{truncation error} of a scheme~\eqref{eq:scheme} on a set of admissible functions $\Phi$ is a function $\tau(h)$ such that for any $\phi\in\Phi$ there exists a constant $C\geq 0$ such that
\[ \abs{F^h(x,\phi(x),\phi(x)-\phi(\cdot)) - F(x,\phi(x),\nabla \phi(x), D^2 \phi(x))} \leq C\tau(h) \]
for every $x\in\G$ and sufficiently small $h>0$.
\end{definition}

\begin{definition}[Monotonicity]\label{def:mono}
The scheme~\eqref{eq:scheme} is \emph{monotone} if $F^h$ is a non decreasing function of its last two arguments. 
\end{definition}

\begin{definition}[Stability]\label{def:stab}
The scheme~\eqref{eq:scheme} is \emph{stable} if there exists some $M>0$, independent of $h$, such that every solution $u^h$ satisfies $||u^h||_{\infty} < M$.
\end{definition}

These simple concepts lead immediately to convergence of finite difference methods, provided the underlying PDE satisfies a strong comparison principle.
\begin{theorem}[Convergence~\cite{BSnum}]\label{thrm:Conv}
Let $u$ be the unique viscosity solution of the PDE~\eqref{eq:elliptic}, where $F$ is a degenerate elliptic operator with a strong comparison principle.   Let $u^h$ be any solution of~\eqref{eq:scheme} where $F^h$ is a consistent, monotone, stable approximation scheme. Then $u^h$ converges uniformly to $u$ as $h\to0$.
\end{theorem}

This result does apply to the \MA equation~\eqref{eq:MA} with Dirichlet boundary conditions~\eqref{eq:dirichlet} under mild assumptions on the data.  However, many other \MA equations of interest do not possess the strong comparison principle required by the theorem.  In recent years, the convergence proof has been adapted to include discontinuous solutions of the non-classical Dirichlet problem~\cite{Hamfeldt_Gauss} and the second boundary value problem~\cite{BenamouDuval_MABVP2,Bonnet_OTBC,Hamfeldt_OTBC}.

\subsection{Wide stencil methods}
Several monotone finite difference approximations have been proposed for the \MA operator~\cite{benamou2014monotone,Bonnet_OTBC,WanMA,FO_MATheory,Nochetto_MAConverge}.  These hinge upon different reformulations of the \MA operator, which typically take a variational form
\bq\label{eq:MAvar}
\det(D^2u) = \min\limits_{(\nu_1, \ldots, \nu_k) \in \Af} G\left(u_{\nu_1\nu_1}, \ldots, u_{\nu_k\nu_k}\right).
\eq
Here $u_{\nu\nu}$ denotes the second directional derivative (SDD) in the direction $\nu\in\R^n$, $\Af$ is some admissible set, and $G$ is a non-decreasing function.  Generating a monotone approximation then requires (1) an appropriate discretization of the relevant SDDs and (2) an appropriate discretization of the admissible set.

On a structured grid, where aligned points $x$, $x+h^+\nu$, and $x-h^-\nu$ are available for some $h^-, h^+>0$, a simple (negative) monotone approximation is
\begin{equation}\label{eq:sdd}
\begin{split}
    \Dt_{\nu\nu} u(x) &\equiv 2\frac{h^-u(x+h^+\nu)+h^+u(x-h^-\nu)-(h^++h^-)u(x)}{h^+h^-(h^++h^-)}\\
      &= u_{\nu\nu}(x) + \bO(h^+-h^-) + \bO{\left({(h^+)^2+(h^-)^2}\right)}.
\end{split}
\end{equation}
These approximations are typically allowed to have a wide-stencil flavor, with the spacing $h^+$, $h^-$ being potentially larger than the characteristic spacing $h$ of grid points.  See Figure~\ref{fig:stencil}.  We also remark that in the special case of equi-spaced neighboring points ($h^+=h^-$), such as on a uniform Cartesian grid, this reduces to the usual centered difference approximation with second order truncation error.  Monotone approximations are also possible on unstructured grids, though they are typically less accurate~\cite{FroeseMeshfreeEigs}.

\begin{figure}
    \centering
    \includegraphics[width=0.35\textwidth, height=0.35\textwidth]{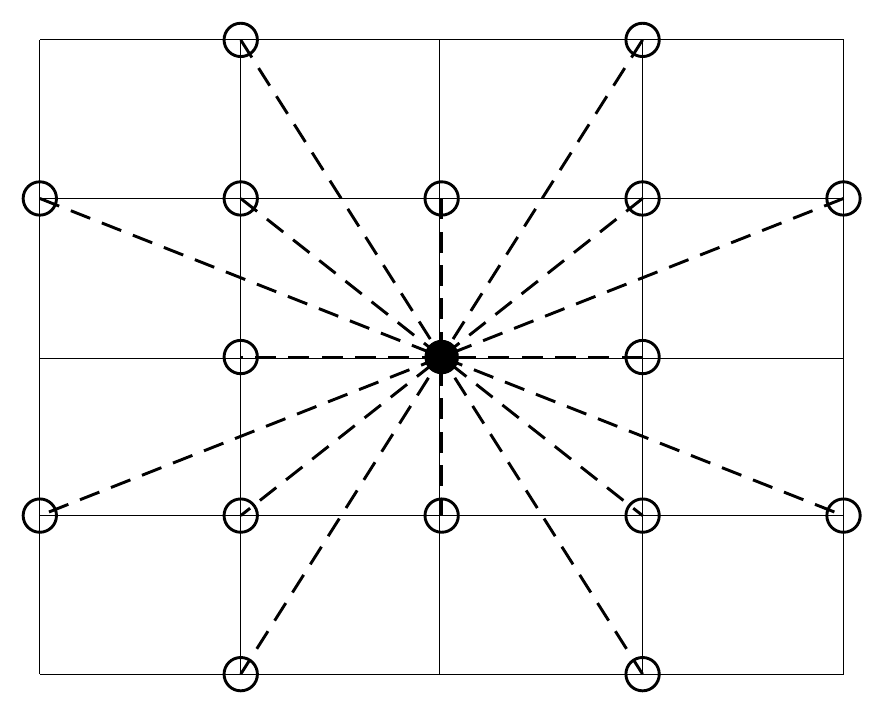}
    \caption{Wide finite difference stencils.}
    \label{fig:stencil}
\end{figure}

The width of stencils required by the approximations of~\eqref{eq:MAvar} is determined by the discretization of the admissible set $\Af$.  Optimal discretization of this set is itself a non-trivial problem and evaluating~\eqref{eq:MAvar} may involve minimization over a prohibitively large set of candidate directions, particularly in three dimensions~\cite{HL_ThreeDimensions}.  A typical scaling for the maximal stencil width that optimizes truncation error is at least $\bO(\sqrt{h})$~\cite{FroeseMeshfreeEigs}, though in some cases it is not clear what the optimal choice is.

\section{Integral Formulation}\label{sec:integral}
In this section, we present an integral representation of $\maop$ which can be used to create a monotone discretizaton through the use of quadrature. 

To motivate this, we recall a well know result about the integrals of multivariate Gaussians:
\begin{equation}
    \det(M) = \pi^{n}\bracs{\int_{\R^n}e^{-v^TMv}dV}^{-2}
\end{equation}
where $M$ is a symmetric positive-definite $n\times n$ matrix. 

This provides an alternate characterization of the \MA operator if we let $M=D^2u(x)$ be the Hessian of the potential function $u$.  Provided $u$ is strictly convex, its Hessian is positive definite and we can write
\bq\label{eq:integral}
\det(D^2u(x)) = \pi^{n}\bracs{\int_{\R^n}e^{-v^TD^2u(x)v}dV}^{-2}.
\eq

To express this in a form that is easily discretized, we convert to spherical coordinates.  Let $r=\abs{v}$ and $\hat{v} = v/r$ and denote by $u_{\hat{v}\hat{v}} = \hat{v}^TD^2u(x)\hat{v}$ the second directional derivative of $u$ in the direction of $\hat{v}$.  Then the \MA operator in $\R^n$ can be expressed as
\bq\label{eq:integralMA}
\det(D^2u(x)) = \pi^{n}\bracs{\int_{\hat{v}\in\mathbb{S}^{n-1}}\int_0^{\infty}r^{n-1}e^{-r^2u_{\hat{v}\hat{v}}}dr d\hat{V}}^{-2}.
\eq
Integrating out the radius $r$, we obtain
\bq\label{eq:integralMARn}
\det(D^2u(x)) = \frac{4\pi^n}{\Gamma(n/2)^2}\bracs{\int_{\mathbb{S}^{n-1}}(u_{\hat{v}\hat{v}})^{-n/2}d\hat{V}}^{-2}.
\eq

For simplicity, the results in the remainder of this paper are presented in two dimensions. However, they are certainly generalizable to higher dimensions. We introduce the notation
\[ u_{\theta\theta} = \frac{\partial^2u}{\partial\nu^2}, \quad \nu = (\cos\theta, \sin\theta) \]
and note that
\[ u_{\theta\theta} = u_{\theta+\pi,\theta+\pi}. \]
Then we easily obtain the two-dimensional version of~\eqref{eq:integralMARn} in terms of polar coordinates. \begin{theorem}[Integral Representation]\label{thrm:integral}
Let $\Omega\subset\R^2$ be convex and $u \in C^2(\Omega)$ be strictly convex.  Then for every $x\in\Omega$:
\begin{equation}\label{eq:IntForm}
    \det(D^2u(x)) = \bracs{\frac{1}{\pi}\int_{0}^{\pi}\frac{d\theta}{\utt(x)}}^{-2}.
\end{equation}
\end{theorem}

The characterization in theorem~\ref{thrm:integral} only holds when $D^2u(x)\succ 0$.  However, we are also interested in degenerate cases where $\det(D^2u(x)) = 0$. In these cases, we know that $D^2u(x)$ has at least eigenvalue equal to zero, and the integrand in~\eqref{eq:IntForm} becomes singular.

We introduce the following relaxation to approximate the integral in these cases:
\begin{equation}
    \det{}_{{\varepsilon_1}}(D^2u(x)) = \bracs{\frac{1}{\pi}\int_{0}^{\pi}\frac{d\theta}{\max(\utt(x),{\varepsilon_1})}}^{-2}.
\end{equation}
This, in turn, is used to construct a relaxed version of the convexified \MA operator:
\begin{equation}
   {\det}^+_{{\varepsilon_1,\epsilon_2}}(D^2u(x)) = \det{}_{{\varepsilon_1}}(D^2u(x)) + \min_{\theta \in [0,\pi)}\{\min(\utt,{\varepsilon_2})\}.
\end{equation}
Here we have represented $\lambda_1(D^2u)$ as $\min\{\utt\}$, which is equivalent via the minimax principle.

\section{Quadrature Scheme}\label{sec:quadScheme}
In this section, we describe a very general framework for utilizing the integral formulation~\eqref{eq:integralMA} to produce a consistent, monotone approximation of the two-dimensional \MA equation.  In \autoref{sec:method}, we will describe two particular implementations.
\subsection{Approximation scheme}
We introduce the following notation.
\begin{definition}[Notation]\label{def:notation}
\end{definition}
\begin{enumerate}
\item[(N1)] $\Omega\subset\R^2$ is a bounded, open, convex domain with Lipschitz boundary $\partial\Omega$.
\item[(N2)] $\G\subset\bar{\Omega}$ is a finite set of discretization points $x_i$, $i=1,\ldots,N$.
\item[(N3)] $h = \sup\limits_{x\in{\Omega}}\min\limits_{y\in\G}\abs{x-y}$ is the spatial resolution of the grid.  In particular, every ball of radius $h$ contained in $\bar{\Omega}$ contains at least one discretization point $x_i$.
\item[(N4)] $r \geq h$ is a stencil width associated to the grid.
\item[(N5)] $0 \leq \theta_0 < \ldots < \theta_M < \pi$ is a finite set of angles discretizing $[0,\pi)$.
\item[(N6)] $d\theta_i = \theta_{i+1}-\theta_i$ is the local angular resolution of the discretization, where we define $d\theta_M = \theta_0 + \pi - \theta_M$.
\item[(N7)] $d\theta = \max\limits_{i=0,\ldots,M}\{d\theta_i\}$ is the angular resolution of the discretization.  
\item[(N8)] $Q = \dfrac{d\theta}{\min\limits_{i=0,\ldots,M} d\theta_i}$ is the \emph{quasi-uniformity constant} of the angular discretization.
\item[(N9)] $w_0, \ldots, w_M$ is a collection of non-negative quadrature weights summing to $\pi$ and satisfying
\[ w_k \geq c d\theta \]
for some constant $c>0$ that depends only on the quasi-uniformity constant.
\item[(N10)] {$\epsilon_1>0$ and $\epsilon_2 \geq 0$ are regularization parameters associated with the grid.}
\item[(N11)] $\Nf(x)\subset\{1, \ldots, N\}$ is the set of neighboring indices for $x\in\G\cap\Omega$ such that for every $j\in\Nf(x)$ we have $0 < \abs{x_j-x} \leq r$.
\item[(N12)] $\Dt_{\theta\theta}u(x)$ described in~\eqref{eq:sdd} has the form
\[ \Dt_{\theta\theta}u(x) = \sum\limits_{j\in\Nf(x)}a_j(\theta)\left(u(x_j)-u(x)\right) \]
for every $x\in\G\cap\Omega$, where all $a_i \geq 0$.
\item[(N13)] $\tau_\theta(r)$ is the truncation error of the finite difference scheme $\Dt_{\theta\theta}u$ for approximating the second directional derivative $u_{\theta\theta}$ on the admissible set $\Phi = C^{2,1}(\Omega)$.
\item[(N14)] $\tau_{FD}(r) = \max\limits_{i=1,\ldots,M} \tau_{\theta_i}(r)$ is the maximal truncation error of the finite difference approximations.
\item[(N15)] $\tau_Q(d\theta)$ is the truncation error of the quadrature scheme 
\[ \sum\limits_{i=0}^M w_if(\theta_i) \]
for approximating the integral $\int_0^\pi f(\theta)\,d\theta$ on the admissible set $\Phi =  \{f\in C^\infty([0,\pi]) \mid f \text{ is periodic}\}$.
\end{enumerate}

Then we propose the following scheme for approximating the convexified \MA operator at interior points $x\in\G\cap\Omega$.
\bq\label{eq:quadscheme}
{G}^h(x,u(x),u(x)-u(\cdot)) =
-\left(\frac{1}{\pi}\sum\limits_{i=0}^M \frac{w_i}{\max\{\Dt_{\theta_i \theta_i}u(x), {\epsilon_1}\}} \right)^{-2}  - \min\limits_{i=0,\ldots,M}\left\{\Dt_{\theta_i\theta_i} u(x),{\epsilon_2}\right\}.
\eq

\subsection{Convergence}
We now provide conditions under which the scheme~\eqref{eq:quadscheme} is consistent and monotone.  As an immediate consequence, it fits directly into the convergence proofs developed in~\cite{BenamouDuval_MABVP2,Bonnet_OTBC,FO_MATheory,Hamfeldt_Gauss,Hamfeldt_OTBC,ObermanEigenvalues}.

\begin{theorem}[Monotonicity]\label{thm:monotonicity}
The approximation scheme~\eqref{eq:quadscheme} is  monotone.
\end{theorem}
\begin{proof}
We note that the operator that appears in~\eqref{eq:quadscheme} can be written in the form
\[ {G}^h(x,u,z) = -\left(\frac{1}{\pi}\mathlarger{\mathlarger{\sum}}\limits_{i=0}^M{\frac{w_i}{\max\{ -\sum\limits_{j\in \Nf(x)} a_{ij}z_j,\, {\epsilon_1} \}}}\right)^{-2} - \min\limits_{i=0, \ldots, M} \left\{-\sum\limits_{j\in \Nf(x)} a_{ij}z_j,\, {\epsilon_2}\right\} \]
where $z_j = u(x)-u(x_j)$ and the $a_{ij}$ are non-negative by the (negative) monotonicity of the approximations $\Dt_{\theta_i\theta_i}u(x)$.

Let $\delta\in\R^N$ have non-negative components.  We notice that
\[ -\sum\limits_{j\in \Nf(x)} a_{ij}(z_j + \delta_j) \leq -\sum\limits_{j\in \Nf(x)} a_{ij}z_j. \]

Since the max and min operators preserve monotonicity and the weights $w_i$ are non-negative, we can immediately conclude that
\[ {G}^h(x,u,z+\delta) \geq {G}^h(x,u,z). \]

Since ${G}^h$ has no dependence on its second argument, this completes the proof of monotonicity.
\end{proof}

\begin{theorem}[Consistency]\label{thm:consistency}
Consider discretizations $\G^h$ of $\bar{\Omega}$ such that the corresponding parameters 
\[  r, d\theta, \frac{{\epsilon_1}}{d\theta}, \frac{\tau_{FD}(r)}{d\theta}, \tau_Q(d\theta), {\epsilon_2} \to 0 \]
as $h\to0$ and the corresponding quasi-uniformity constants $Q$ are bounded uniformly.
Then the approximation scheme~\eqref{eq:quadscheme} is consistent with the convexified \MA operator~\eqref{eq:MAconvex}.
\end{theorem}

We will break this result into three separate cases (Lemmas~\ref{lem:consistentPositive}-\ref{lem:consistentZero}), depending on the sign of $\lambda_1(D^2u)$, the smallest eigenvalue of the Hessian.  We note that the scheme ${G}^h(x,u,z)$ appearing in~\eqref{eq:quadscheme}  has no dependence on the first two arguments, which allows us to simplify slightly the verification of consistency.

\begin{lemma}[Consistency with positive eigenvalues]\label{lem:consistentPositive}
Under the assumptions of Theorem~\ref{thm:consistency}, let $u\in C^{2,1}(\Omega)$ and consider  $x\in\Omega$ such that $\lambda_1(D^2u(x)) > 0$.  Then the scheme~\eqref{eq:quadscheme} satisfies
\[ \lim\limits_{y\in\G\to x, h\to0} {G}^h(y,u(y),u(y)-u(\cdot)) = -{\det}^+(D^2u(x)). \]
\end{lemma}
\begin{proof}
Since the smallest eigenvalue $\lambda_1(D^2u(x))$ is strictly positive and $u\in C^{2,1}$, we are assured that
\[ \lambda_1(D^2u(y)) > \frac{1}{2}\lambda_1(D^2u(x)) > {\epsilon_k} + {\bO}(\tau_{FD}(r)), {\quad k \in \{1, 2\}} \]
for all $y$ sufficiently close to $x$ and sufficiently small ${\epsilon_1, \epsilon_2, r}$.  Then using the consistency error in the components of this scheme, we can compute
\begin{align*}
{G}^h(y,&u(y),u(y)-u(\cdot)) 
  = -\left(\frac{1}{\pi}\sum\limits_{i=0}^M \frac{w_i}{\max\{\Dt_{\theta_i \theta_i}u(y), {\epsilon_1}\}} \right)^{-2}  - \min\limits_{i=0,\ldots,M}\left\{\Dt_{\theta_i\theta_i} u(y),{\epsilon_2}\right\}\\
  &= -\left(\frac{1}{\pi}\sum\limits_{i=0}^M \frac{w_i}{\max\{u_{\theta_i\theta_i}(y) + {\bO}(\tau_{FD}(r)),{\epsilon_1}\}} \right)^{-2}  - \min\limits_{i=0,\ldots,M}\left\{u_{\theta_i\theta_i}(y) + {\bO}(\tau_{FD}(r)),{\epsilon_2}\right\}\\
  &= -\left(\frac{1}{\pi}\sum\limits_{i=0}^M \frac{w_i}{u_{\theta_i\theta_i}(y) + {\bO}(\tau_{FD}(r))} \right)^{-2}  - {\epsilon_2}\\
  &= -\left(\frac{1}{\pi}\sum\limits_{i=0}^M \frac{w_i}{u_{\theta_i\theta_i}(y)} + \bO(\tau_{FD}(r))\sum\limits_{i=0}^M\frac{w_i}{u_{\theta_i\theta_i}(y)^2} \right)^{-2}  - {\epsilon_2}.
\end{align*}

Since $u_{\theta_i\theta_i}(y)$ is continuous in $\theta$ and bounded away from zero, the two sums in the last line can both be interpreted as consistent quadrature schemes.  Thus we can further estimate
\begin{align*}
{G}^h(y,&u(y),u(y)-u(\cdot)) = -\left(\frac{1}{\pi}\int_0^\pi \frac{1}{u_{\theta\theta}(y)}\,d\theta + \bO(\tau_Q(d\theta)+\tau_{FD}(r)) \right)^{-2}  - {\epsilon_2}.
\end{align*}

Recalling now the integral formulation of the \MA operator~\eqref{eq:IntForm}, we conclude that
\[
{G}^h(y,u(y),u(y)-u(\cdot)) = -\det(D^2u(y)) + \bO(\tau_Q(d\theta)+\tau_{FD}(r) + {\epsilon_2}).
\]

Since $u\in C^{2,1}$ and all eigenvalues of the Hessian are strictly positive, we conclude that
\[ \lim\limits_{y\in\G\to x, h\to0}{G}^h(y,u(y),u(y)-u(\cdot)) = -{\det}^+(D^2u(x)). \]
\end{proof}

\begin{lemma}[Consistency with a negative eigenvalue]\label{lem:consistentNegative}
Under the assumptions of Theorem~\ref{thm:consistency}, let $u\in C^{2,1}(\Omega)$ and consider  $x\in\Omega$ such that $\lambda_1(D^2u(x)) < 0$.  Then the scheme~\eqref{eq:quadscheme} satisfies
\[ \lim\limits_{y\in\G\to x, h\to0} {G}^h(y,u(y),u(y)-u(\cdot)) = -{\det}^+(D^2u(x)). \]
\end{lemma}
\begin{proof}

Suppose without loss of generality that the coordinates are chosen so that the eigenvector corresponding to the smallest eigenvalue $\lambda_1(D^2u(y))$ is $v_1 = (1,0)$.  Then the second directional derivative of $u$ in the direction $\theta$ can be expressed as
\[ u_{\theta\theta}(y) = \lambda_1(y)\cos^2\theta + \lambda_2(y)\sin^2\theta.\]

Let us consider in particular the first angle $\theta_0 \leq d\theta$ in the angular discretization.  Since $u\in C^{2,1}$, the second directional derivative of $u$ in this direction satisfies
\begin{align*}
u_{\theta_0\theta_0}(y) &= \lambda_1(D^2u(y)) + \bO(d\theta^2) \\
  &= \lambda_1(D^2u(x)) + \bO(d\theta^2 + \abs{x-y}).
\end{align*}
Since $\lambda_1(D^2u(x))$ is strictly negative, it is certainly the case that for  $y$ sufficiently close to $x$ and small enough $r, d\theta$:
\begin{align*}
\Dt_{\theta_0\theta_0}u(y) &= u_{\theta_0\theta_0}(y) + {\bO}(\tau_{FD}(r))\\
&= \lambda_1(D^2u(x)) + \bO(\tau_{FD}(r) + d\theta^2 + \abs{x-y})\\
&{<0} \\ 
&< {\epsilon_k}, & {k \in \{1,2\}}.
\end{align*}

Now we perform a crude estimate on the sum in~\eqref{eq:quadscheme} by considering only a single term:
\begin{align*}
0 &\leq \left(\frac{1}{\pi}\sum\limits_{i=0}^M\frac{w_i}{\max\{\Dt_{\theta_i\theta_i}u(y),{\epsilon_1}\}}\right)^{-2}\\
  &\leq \left(\frac{1}{\pi}\frac{w_0}{\max\{\Dt_{\theta_0\theta_0}u(y),{\epsilon_1}\}}\right)^{-2}\\
  &= \frac{\pi^2{\epsilon_1}^2}{w_0^2}\\
  &\leq \frac{\pi^2{\epsilon_1}^2}{c^2d\theta^2}.
\end{align*}

Using the same estimates on the discrete second directional derivatives, we can also estimate the term
\begin{align*}
\min\limits_{i=0,\ldots,M}\{\Dt_{\theta_i\theta_i},{\epsilon_2}\} &= \min\{\lambda_1(D^2u(x)) + \bO(\tau_{FD}(r) + d\theta^2 + \abs{x-y}), {\epsilon_2}\}\\
  &= \lambda_1(D^2u(x)) + \bO(\tau_{FD}(r) + d\theta^2 + \abs{x-y}).
\end{align*}

By combining these estimates and recalling that ${\epsilon_1}/d\theta\to0$, we conclude that
\begin{align*}
\lim\limits_{y\in\G, h \to 0}{G}^h(y,u(y),u(y)-u(\cdot)) 
  &= -\lambda_1(D^2u(x))\\
  &= -{\det}^+(D^2u(x)).
\end{align*}
\end{proof}

\begin{lemma}[Consistency with a vanishing eigenvalue]\label{lem:consistentZero}
Under the assumptions of Theorem~\ref{thm:consistency}, let $u\in C^{2,1}(\Omega)$ and consider  $x\in\Omega$ such that $\lambda_1(D^2u(x))= 0$.  Then the scheme~\eqref{eq:quadscheme} satisfies
\[ \lim\limits_{y\in\G\to x, h\to0} {G}^h(y,u(y),u(y)-u(\cdot)) = -{\det}^+(D^2u(x)). \]
\end{lemma}
\begin{proof}
By the regularity of $u$, we know that
\[ \lambda_1(D^2u(y)) = \bO(\abs{x-y}). \]

Suppose without loss of generality that the coordinates are chosen so that the eigenvector corresponding to the smallest eigenvalue $\lambda_1(D^2u(y))$ is $v_1 = (1,0)$.  Then the second directional derivative of $u$ in the direction $\theta$ can be expressed as
\[ u_{\theta\theta}(y) = \lambda_1(y)\cos^2\theta + \lambda_2(y)\sin^2\theta.\]

Now we are going to estimate the sum in~\eqref{eq:quadscheme} by considering only the angles $\theta$ that are close to zero, which corresponds to the direction of the eigenvector $v_1$.  To this end, we define
\[ s = \max\{d\theta, \sqrt{\abs{x-y}}\}\]
and let $K = \bO(s/d\theta) \geq 1$ be the number of nodes $\theta_0, \ldots, \theta_{K-1}$ in the interval $[0,s]$.

We notice that for any $i = 0, \ldots, K-1$ we have
\begin{align*}
\max\{\Dt_{\theta_i\theta_i}u(y),{\epsilon_1}\} &\leq u_{\theta_i\theta_i}(y) + \tau_{FD}(r) + {\epsilon_1}\\
 &= \lambda_1(y)\cos^2\theta_i + \lambda_2(y)\sin^2\theta_i + \tau_{FD}(r) + {\epsilon_1}\\
 &= \bO(\abs{x-y} + s^2 + \tau_{FD}(r) + {\epsilon_1}).
\end{align*}

Then using the lower bound $w_i > c \,d\theta$ allows us to obtain the following bounds on the sum.
\begin{align*}
 \sum\limits_{i=0}^M\frac{w_i}{\max\{\Dt_{\theta_i\theta_i}u(y),{\epsilon_1}\}}
  &\geq \sum\limits_{i=0}^{K-1}\frac{w_i}{\max\{\Dt_{\theta_i\theta_i}u(y),{\epsilon_1}\}}\\
  &\geq  K\frac{c\,d\theta}{\bO(\abs{x-y} + s^2 + \tau_{FD}(r) + {\epsilon_1})}\\
  &= \frac{s}{d\theta}\frac{c\,d\theta}{\bO(\abs{x-y} + s^2 + \tau_{FD}(r) + {\epsilon_1})}\\
  &= \frac{cs}{\bO(\abs{x-y} + s^2 + \tau_{FD}(r) + {\epsilon_1})}.
\end{align*}

This allows us to obtain bounds on the following value appearing in the scheme:
\begin{align*}
\left(\frac{1}{\pi}\sum\limits_{i=0}^M \frac{w_i}{\max\{\Dt_{\theta_i \theta_i}u(y), {\epsilon_1}\}} \right)^{-2}
  &\leq \bO\left(\frac{\abs{x-y}^2}{s^2} + \frac{s^4}{s^2} + \frac{\tau_{FD}(r)^2}{s^2} + \frac{{\epsilon_1}^2}{s^2}\right).
\end{align*}

Recalling from the definition of $s$ that $s\geq d\theta$ and $s \geq \sqrt{\abs{x-y}}$
allows us to simplify this as follows:
\begin{align*}
0 &\leq
\left(\frac{1}{\pi}\sum\limits_{i=0}^M \frac{w_i}{\max\{\Dt_{\theta_i \theta_i}u(y), {\epsilon_1}\}} \right)^{-2}\\
&\leq \bO\left(\frac{\abs{x-y}^2}{\abs{x-y}} + \max\{d\theta^2,\abs{x-y}\} + \frac{\tau_{FD}(r)^2}{d\theta^2} + \frac{{\epsilon_1}^2}{d\theta^2}\right).
\end{align*}
Thus under the conditions of Theorem~\ref{thm:consistency}, we find that
\[ \lim\limits_{y\in\G\to x, h \to 0}\left(\frac{1}{\pi}\sum\limits_{i=0}^M \frac{w_i}{\max\{\Dt_{\theta_i \theta_i}u(y), {\epsilon_1}\}} \right)^{-2} = 0. \]

We also observe that
\begin{align*}
{\epsilon_2}
&\geq \min\limits_{i=0,\ldots,M}\left\{\Dt_{\theta_i\theta_i} u(y),{\epsilon_2}\right\}\\
&\geq \min\limits_{i=0,\ldots,M}\{\lambda_1(D^2u(y))\cos^2\theta_i + \lambda_2(D^2u(y))\sin^2\theta_i - \bO(\tau_{FD}(r)), {\epsilon_2}\}\\
&\geq  \min\limits_{i=0,\ldots,M}\{-\bO(\abs{x-y})\cos^2\theta_i + \lambda_2(D^2u(y))\sin^2\theta_i - \bO(\tau_{FD}(r)),{\epsilon_2}\}\\
&\geq -\bO(\abs{x-y} + \tau_{FD}(r))
\end{align*}
which implies that
\[ \lim\limits_{y\in\G\to x, h \to 0} \min\limits_{i=0,\ldots,M}\left\{\Dt_{\theta_i\theta_i} u(y),{\epsilon_2}\right\} = 0.\]

We conclude that
\[  \lim\limits_{y\in\G\to x, h \to 0}{G}^h(y,u(y),u(y)-u(\cdot)) = 0,\]
which coincides with the value of \[{\det}^+(D^2u(x)) = \lambda_1(D^2u(x)) \lambda_2(D^2u(x)) = 0.\]
\end{proof}

\subsection{Quadrature rules}
In designing a scheme of the form~\eqref{eq:quadscheme}, the choice of quadrature rule
\[ \sum\limits_{i=0}^Mw_i f(\theta_i) \approx \int_0^\pi f(\theta)\,d\theta \]
is a key factor that will influence the overall cost and accuracy.

A simple choice is the trapezoid rule, which utilizes the weights
\begin{equation}\label{eq:trap}
    w_i = \begin{cases} 
    \dfrac{\theta_0 + \pi-\theta_M}{2}, \quad & i = 0
    \\ 
    \dfrac{\theta_{i+1}-\theta_{i-1}}{2}, & i = 1,...,M-1
    \\
    \dfrac{\pi-\theta_{M-1}+\theta_0}{2}, & i = M.
    \end{cases}
\end{equation}

As required, the weights are all positive.  As required by (N8) (Definition~\ref{def:notation}), they can also be bounded from below in terms of the quasi-uniformity constant via
\[ w_i = \frac{d\theta_{i-1}+d\theta_i}{2} \geq \frac{d\theta}{Q}. \]

In general, the truncation error of the trapezoid rule is $\tau_Q(d\theta) = d\theta^2$.  However, in the special case of a uniform angular discretization ($d\theta_i = d\theta$ for all $i = 0, \ldots, M$), the trapezoid rule is spectrally accurate.  In this case, the truncation error satisfies $\tau_Q(d\theta) \leq d\theta^p$ for every $p>0$ given a $C^\infty$ integrand.

Higher-order quadrature is also possible {on non-uniform angular discretizations}.  In fact, as we will demonstrate in \autoref{sec:method}, this can be exploited in order to design approximation schemes that simultaneously improve the formal consistency error and reduce the required stencil width.

As an example, we consider Simpson's rule.  Suppose that $M+1$, the number of angles in the angular discretization, is even.  Then Simpson's rule takes the form
\begin{equation}\label{eq:simp}\begin{split}
    \int_0^\pi f(\theta)\,d\theta \approx \sum_{i = 0}^{(M-1)/2}&\frac{d\theta_{2i}+d\theta_{2i+1}}{6}\bigg{[}\bigg{(}2-\frac{d\theta_{2i+1}}{d\theta_{2i}}\bigg{)}f(\theta_{2i})\\&+\dfrac{(d\theta_{2i}+d\theta_{2i+1})^2}{d\theta_{2i}d\theta_{2i+1}}f(\theta_{2i+1})+\bigg{(}2-\dfrac{d\theta_{2i}}{d\theta_{2i+1}}\bigg{)}f(\theta_{2i+2})\bigg{]}
\end{split}\end{equation}
where we identify $\theta_{j+M+1}={\theta_j+\pi}$ and $d\theta_j = d\theta_{j+M+1}$ because of the periodicity of $f$.

Rearranging, we find that the corresponding quadrature weights are
\bq\label{eq:simpsonW}
w_j = \begin{cases}
\dfrac{(d\theta_{j-1}+d\theta_j)^3}{6d\theta_{j-1}d\theta_j}, & j \text{ odd}\\
\dfrac{d\theta_j+d\theta_{j+1}}{6}\left(2-\dfrac{d\theta_{j+1}}{d\theta_j}\right) +\dfrac{d\theta_{j-2}+d\theta_{j-1}}{6}\left(2-\dfrac{d\theta_{j-2}}{d\theta_{j-1}}\right), & j \text{ even}.
\end{cases}
\eq

The truncation error associated with Simpson's rule is $\tau_Q(d\theta) = d\theta^4$.  

However, unlike with the trapezoid rule, these quadrature weights are not automatically positive.  Instead, positivity is guaranteed only if the quasi-uniformity constant of the angular discretization is not too large.  In particular, we note that $Q<2$ is sufficient to guarantee that
\[ 2-\frac{d\theta_{j\pm1}}{d\theta_j} \geq 2 - \frac{d\theta}{d\theta_j} \geq 2-Q > 0. \]

Under the same assumption on quasi-uniformity, we use the fact that
\[ \frac{d\theta}{Q} \leq d\theta_j \leq d\theta \]
to verify that
\[ w_j \geq \min\left\{\frac{(2d\theta / Q)^3}{6d\theta^2}, 2\frac{2d\theta/Q}{6}\left(2-\frac{d\theta}{d\theta/Q}\right)\right\} = \min\left\{\frac{4}{3Q^3}, \frac{2}{3Q}\left(2-Q\right)\right\} d\theta,\]
as required by (N8) (Definition~\ref{def:notation}).

Similar results can be obtained using other higher-order quadrature schemes, which will place differing requirements on the quasi-uniformity constant $Q$ in order to ensure positivity of the weights.

\subsection{{Truncation error}}\label{sec:error}
As an immediate consequence of the consistency proofs (in particular, Lemma~\ref{lem:consistentPositive}), we obtain the formal truncation error of the scheme as points where the function is ``locally'' strictly convex.  This will be used to inform and optimize the particular implementations of this method in \autoref{sec:method}.

\begin{corollary}[Truncation error]\label{cor:truncation}
Under the assumptions of Theorem~\ref{thm:consistency}, let $u\in C^{2,1}(\Omega)$ and consider  $x\in\Omega$ such that $\lambda_1(D^2u(x)) > 0$.  Then there exists a constant $C>0$ such that for all sufficiently small $h>0$,
\[ \abs{{G}^h(x,u(x),u(x)-u(\cdot)) + {\det}^+(D^2u(x))} \leq C\left(\tau_Q(d\theta) + \tau_{FD}(r) + {\epsilon_2}\right). \]
\end{corollary}

We are also interested in approximating functions that are convex, but not necessarily strictly convex.  In this case, the integrand in~\eqref{eq:IntForm} is singular and we cannot directly use the formal truncation error $\tau_Q(d\theta)$ of the quadrature rule.  However, we can easily bound the resulting sums directly in the case where at least one eigenvalue $\lambda_1(D^2u(x))$ is known to vanish.  We consider two separate cases: (1) the fully degenerate case ($\lambda_1(D^2u(x)) = \lambda_2(D^2u(x)) = 0$) and (2) the semi-degenerate case ($\lambda_1(D^2u(x)) = 0 < \lambda_2(D^2u(x))$).

\begin{lemma}[Truncation error (fully degenerate)]\label{lem:truncationDeg}
Under the assumptions of Theorem~\ref{thm:consistency}, let $u\in C^{2,1}(\Omega)$ and consider  $x\in\Omega$ such that $\lambda_1(D^2u(x)) = \lambda_2(D^2u(x)) =  0$.  Then there exists a constant $C>0$ such that for all sufficiently small $h>0$,
\[ \abs{{G}^h(x,u(x),u(x)-u(\cdot)) + {\det}^+(D^2u(x))} \leq C\left(\tau_{FD}(r)^2 + {\epsilon_1}\right). \]
\end{lemma}

\begin{lemma}[Truncation error (semi-degenerate)]\label{lem:truncationSemiDeg}
Under the assumptions of Theorem~\ref{thm:consistency}, let $u\in C^{2,1}(\Omega)$ and consider  $x\in\Omega$ such that $\lambda_1(D^2u(x)) = 0 < \lambda_2(D^2u(x))$.  Then there exists a constant $C>0$ such that for all sufficiently small $h>0$,
\[ \abs{{G}^h(x,u(x),u(x)-u(\cdot)) + {\det}^+(D^2u(x))} \leq C\left(d\theta^2 + \frac{\tau_{FD}(r)^2}{d\theta^2} + \frac{{\epsilon_1}^2}{d\theta^2} + {\epsilon_2}\right). \]
\end{lemma}

{
Finally, we observe that with appropriate symmetry in the discretization, our quadrature-based schemes can sometimes result in even better formal consistency error than that predicted by Corollary~\ref{cor:truncation}.  This observation motivates one of the implementations (on hexagonal grids) that will be introduced in section~\ref{sec:method}.

We consider the special case of applying the trapezoid rule using equally spaced angles ($d\theta_i = d\theta$ for all $i=0, \ldots, M$), which is spectrally accurate as discussed previously.  We suppose that we use grid-aligned differences, which may be centered or uncentered, to discretize the finite difference operators.  That is, the error in~\eqref{eq:sdd} takes the form
\[ \Dt_{\theta_i\theta_i}u = u_{\theta_i\theta_i} + \frac{1}{3}u_{\theta_i\theta_i\theta_i}(r(\theta_i)-r(\theta_{i+\pi})) + \bO(r^2). \]
where $\abs{r(\theta_i)} \leq r$ for all $i=0,\ldots,M$.  Notice that by symmetry and periodicity, we have that
\[ u_{\theta+\pi,\theta+\pi} = u_{\theta\theta}, \quad u_{\theta+\pi,\theta+\pi,\theta+\pi} = -u_{\theta\theta\theta}, \quad r(\theta+2\pi) = r(\theta). \]
For ease of notation, we extend the indexing such that $\theta_{i+M+1} = \theta_i + \pi$ for $i=0, \ldots, M$.

\begin{figure}
    \centering
    \subfigure[]{
    \includegraphics[width=0.32\textwidth]{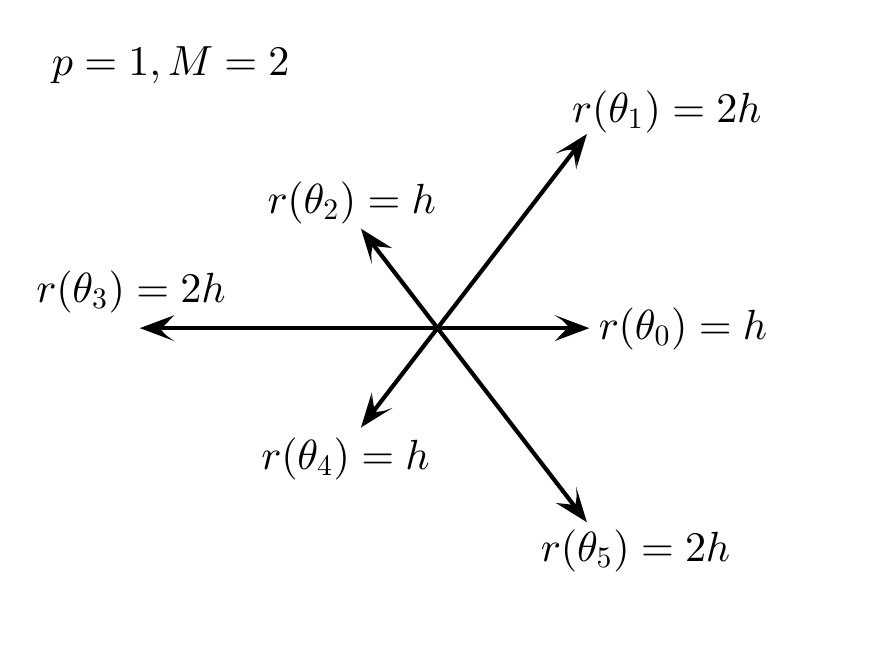}
		\includegraphics[width=0.32\textwidth]{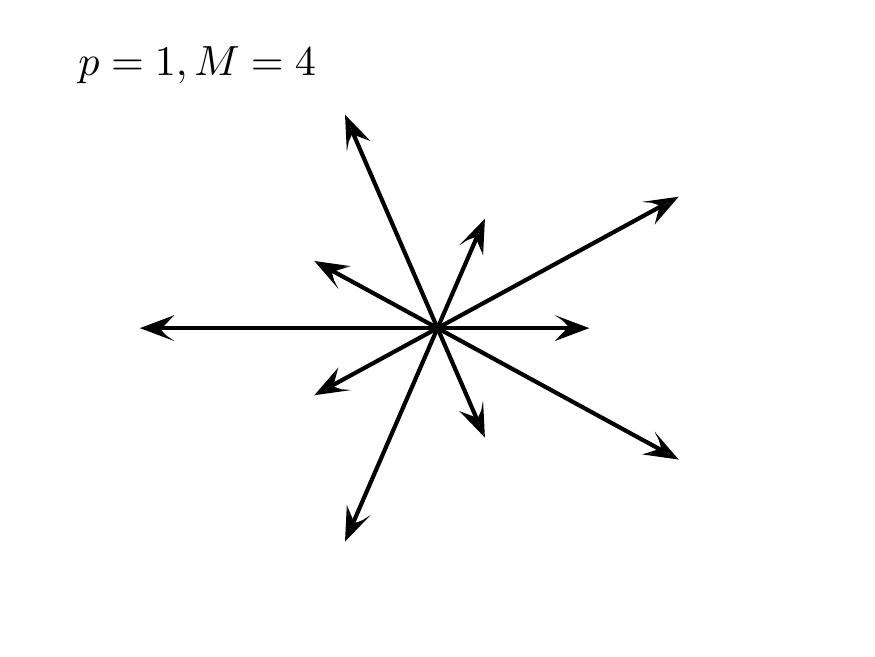}
		\includegraphics[width=0.32\textwidth]{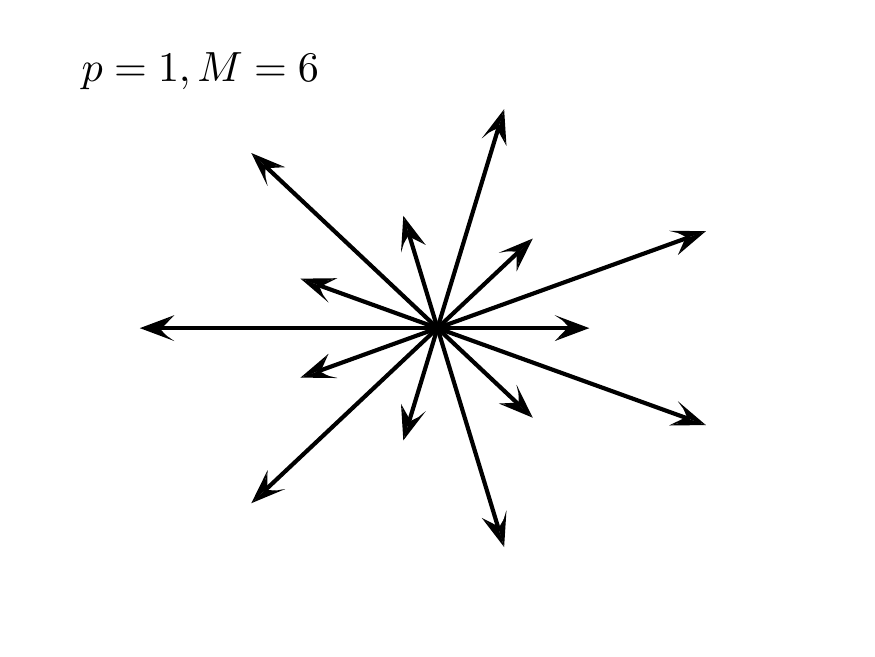}\label{fig:stencilp1}}
    \subfigure[]{
    \includegraphics[width=0.32\textwidth]{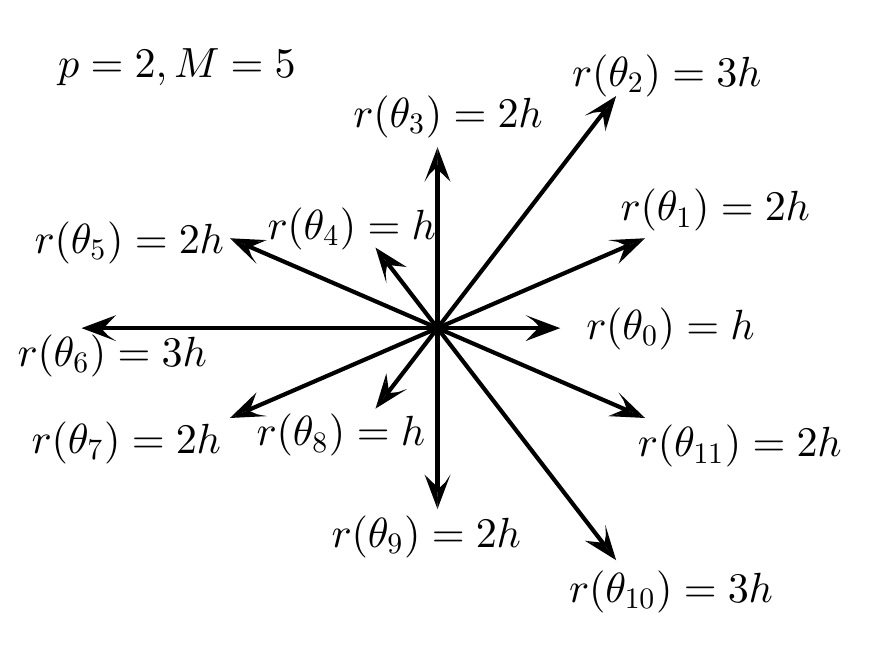}
		\includegraphics[width=0.32\textwidth]{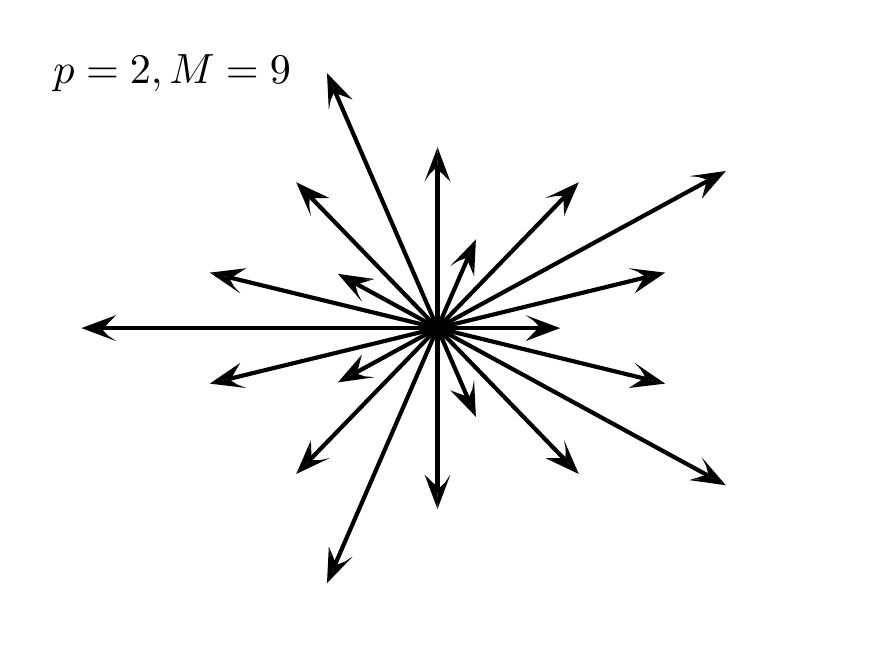}
		\includegraphics[width=0.32\textwidth]{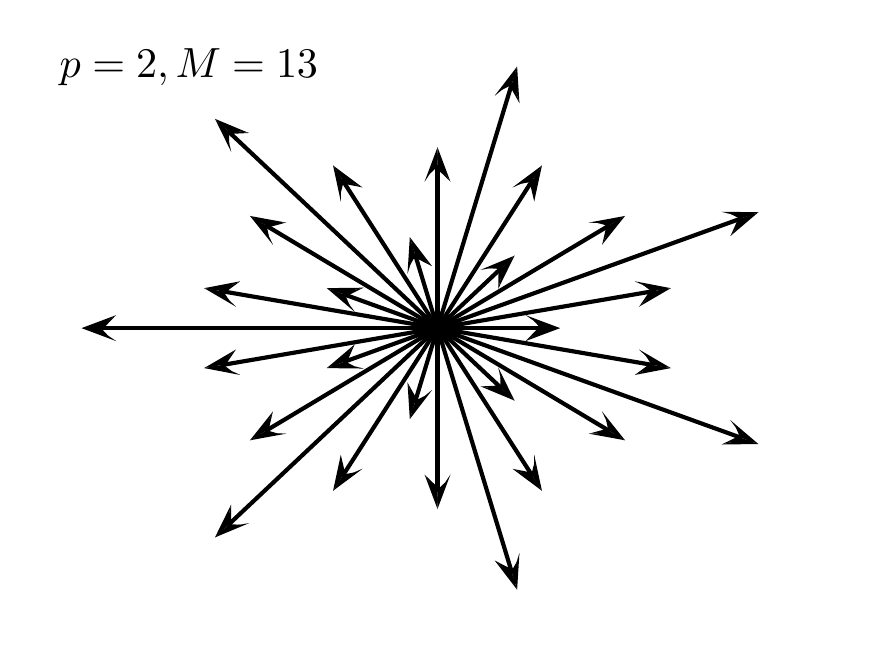}\label{fig:stencilp2}}
    \caption{{Examples of sequences of stencils corresponding to \subref{fig:stencilp1}~$p=1$ and \subref{fig:stencilp2}~$p=2$.}}
    \label{fig:stencilsymmetry}
\end{figure}

Now we assume there is sufficient symmetry in the grids $\G^h$ so that, for some fixed $p\in\mathbb{N}$ and every $h>0$, we have $r(\theta_{i+p})-r(\theta_{i+p}+\pi) = -(r(\theta_i)-r(\theta_{i}+\pi))$.  Note that this condition requires that $M+1$, the number of terms in the angular discretization of $[0,\pi)$, is an odd multiple of $p$ for each grid $\G^h$ and corresponding stencil.  See Figure~\ref{fig:stencilsymmetry}. We argue that if $\lambda_1(D^2u)>0$, we can expect the formal truncation error of~\eqref{eq:quadscheme} to be $\bO(\tau_Q(d\theta) + r^2 + \epsilon_2)$ despite the fact that the underlying finite difference approximations have only $\bO(r)$ accuracy.

We first note that under this symmetry condition, there are at most $p$ possible values that $\abs{r(\theta_{i})-r(\theta_{i}+\pi)}$ can take.  To access these, we rewrite the indices $i = 0, \ldots, 2M+1$ in the form $i = kp+j$ where $j=0,\ldots,p-1$ and $k=0,\ldots,(2M+2)/p-1$.  Then by $p$-periodicity, we find that
\[ \abs{r(\theta_{kp+j})-r(\theta_{kp+j}+\pi)} = \abs{r(\theta_{j})-r(\theta_{j}+\pi)}. \]
Because the sign alternates every $p$ steps, we can further characterize
\[ r(\theta_{kp+j})-r(\theta_{kp+j}+\pi) = (-1)^k(r(\theta_{j})-r(\theta_{j}+\pi)), \]
which can take at most $2p$ distinct values.

In the setting $\lambda_1(D^2u)>0$ (so that all $u_{\theta\theta}\geq\lambda_1(D^2u)>0$), we have that for sufficiently small $r, \epsilon > 0$, the summation appearing in~\eqref{eq:quadscheme} can be expressed as
\begin{align*}
\sum\limits_{i=0}^M &\frac{w_i}{\max\{\Dt_{\theta_i\theta_i}u,\epsilon_1\}} 
  = \frac{\pi}{M+1} \sum\limits_{i=0}^M \frac{1}{u_{\theta_i\theta_i} + \frac{1}{3}u_{\theta_i\theta_i\theta_i}(r(\theta_i)-r(\theta_{i+\pi})) + \bO(r^2)}\\
	&=\frac{\pi}{M+1}\sum\limits_{i=0}^M\frac{1}{u_{\theta_i\theta_i}} - \frac{\pi}{3(M+1)}\sum\limits_{i=0}^M \frac{u_{\theta_i\theta_i\theta_i}(r(\theta_i)-r(\theta_i+\pi))}{u_{\theta_i\theta_i}} + \bO(r^2),
\end{align*}
where the last line here follows from a binomial expansion.

Now we exploit symmetry and the fact that $1/u_{\theta\theta}$ is smooth (since $\lambda_1(D^2u)>0$) to re-express this as
\begin{align*}
\sum\limits_{i=0}^M &\frac{w_i}{\max\{\Dt_{\theta_i\theta_i}u,\epsilon_1\}} \\
&= \int_0^\pi 
	\frac{1}{u_{\theta\theta}}d\theta + \bO(\tau_Q(d\theta) + r^2) - \frac{\pi}{6(M+1)} \sum\limits_{i=0}^{M}\left(\frac{u_{\theta_i\theta_i\theta_i}(r(\theta_i)-r(\theta_i+\pi))}{u_{\theta_i\theta_i}}  \right.\\
	&\phantom{===================}\left. + \frac{-u_{\theta_i+\pi,\theta_i+\pi,\theta_i+\pi}(r(\theta_i+2\pi)-r(\theta_i+\pi))}{u_{\theta_i+\pi,\theta_i+\pi}}\right)\\
  &= \int_0^\pi 
	\frac{1}{u_{\theta\theta}}d\theta + \bO(\tau_Q(d\theta)+r^2) - \frac{\pi}{6(M+1)} \sum\limits_{i=0}^{2M+1} 
	\frac{u_{\theta_i\theta_i\theta_i}(r(\theta_i)-r(\theta_i+\pi))}{u_{\theta_i\theta_i}}.
\end{align*}

Next, we utilize the periodicity of the terms $r(\theta_i)-r(\theta_i+\pi)$ with respect to shifts of $2p$ in the index.  This allows us to rewrite the sum as
\begin{align*}
\sum\limits_{i=0}^M &\frac{w_i}{\max\{\Dt_{\theta_i\theta_i}u,\epsilon_1\}} \\
  &= \int_0^\pi 
	\frac{1}{u_{\theta\theta}}d\theta + \bO(\tau_Q(d\theta)+r^2) \\
	&\phantom{===}- \frac{\pi}{6(M+1)}\sum\limits_{j=0}^{p-1}\sum\limits_{k=0}^{(2M+2)/p-1} \frac{u_{\theta_{pk+j}\theta_{pk+j}\theta_{pk+j}}(-1)^k(r(\theta_j)-r(\theta_j+\pi))}{u_{\theta_{pk+j}\theta_{pk+j}}} \\
	&= \int_0^\pi 
	\frac{1}{u_{\theta\theta}}d\theta + \bO(\tau_Q(d\theta)+r^2)\\
	&\phantom{===}- \frac{1}{6p}\sum\limits_{j=0}^{p-1} (r(\theta_j)-r(\theta_j+\pi))
	\left(\frac{\pi p}{M+1}\sum\limits_{k \text{ even}} \frac{u_{\theta_{pk+j}\theta_{pk+j}\theta_{pk+j}}}{u_{\theta_{pk+j}\theta_{pk+j}}}\right.\\
	&\phantom{=========================}\left.- \frac{\pi p}{M+1}\sum\limits_{k \text{ odd}} \frac{u_{\theta_{pk+j}\theta_{pk+j}\theta_{pk+j}}}{u_{\theta_{pk+j}\theta_{pk+j}}}\right). 
\end{align*}

Now we notice that each of the sums in the last line can be interpreted as the trapezoid rule applied to integral
\[ \int_0^{2\pi}\frac{u_{\theta\theta\theta}}{u_{\theta\theta}}d\theta=0 \]
 using equally spaced angles with $d\tilde{\theta} = 2p\,d\theta$.  Thus formally, we expect that
\[ \sum\limits_{i=0}^M \frac{w_i}{\max\{\Dt_{\theta_i\theta_i}u,\epsilon_1\}} 
  = \int_0^\pi 
	\frac{1}{u_{\theta\theta}}d\theta + \bO(\tau_Q(d\theta)+r^2 + r\,\tau_Q(d\theta)). \]

Substituting this into the quadrature scheme~\eqref{eq:quadscheme} for the \MA equation, we obtain an expected consistency error of $\bO(\tau_Q(d\theta) + r^2 + \epsilon_2)$, which is better than the truncation error predicted by Corollary~\ref{cor:truncation}.
}

\section{Implementation}\label{sec:method}
We now present {three} specific implementations of a quadrature scheme based upon the formulation of~\eqref{eq:quadscheme}.  

{The first two implementations rely on a hexagonal and triangular tiling of the domain, respectively.}  The underlying structure of the grid allows us to design an angular discretization of $[0,2\pi]$ involving twelve equi-spaced angles.  The use of the trapezoid rule then leads to a compact finite difference stencil that in practice achieves spectral {accuracy} in the angular parameter $d\theta$ {(which is held fixed)}.

The {third} implementation relies on a simple Cartesian grid combined with Simpson's Rule for quadrature.  As required by the convergence analysis (Theorem~\ref{thm:monotonicity}-\ref{thm:consistency}), the stencil does grow wider as the grid is refined.  However, the optimal stencil width is asymptotically narrower than that required by existing monotone schemes, while simultaneously improving the order of the formal consistency error~\cite{FroeseMeshfreeEigs}.

\subsection{Discretization of domain}
The implementations we describe rely on a discretization of the domain that consists of two components: (1) a structured mesh restricted to the interior of the domain and (2) a list of boundary points chosen to preserve the desired angular resolution.  Hand-in-hand with this grid, we include the list of angles $\theta_j$ used to discretize the integral in~\eqref{eq:IntForm}.

We begin by presenting an algorithm for discretizing the domain, which applies to {all} types of mesh structure (Hexagonal, {Triangular}, and Cartesian) considered in this work.  In the subsequent subsections, we will fill in the remaining details about each specific implementation of the quadrature scheme.

As a starting point, suppose we are given a structured mesh $\mathcal{M}$ that tiles $\R^2$ and a set of angles $0 \leq \theta_0 < \ldots < \theta_M < \pi$.  Moreover, the angles are chosen such that for every $x\in\mathcal{M}$ and $j = 0, \ldots, M$, we have
\[ x \pm r^\pm_j(x)(\cos\theta_j,\sin\theta_j) \in \mathcal{M} \]
for some $r^\pm_j(x)>0$.  That is, we are able to identify neighboring grid points aligned with all the directions in our angular discretization.  To this underlying grid, we associate a stencil width defined by
\[ r = \max\{r^\pm_j(x) \mid x\in\mathcal{M}; j = 0, \ldots, M\}. \]

From this tiling of $\R^2$, we generate a set of discretization points $\G$ by (1) including all mesh points lying in the interior of the domain $\Omega$ and (2) supplementing with points in $\partial\Omega$ in order to preserve the existence of grid points perfectly aligned with the given set of angles.  That is, given any interior node $x\in\G\cap\Omega$  and $j = 0, \ldots, M$, we have
\[ x \pm r^\pm_j(x)(\cos\theta_j,\sin\theta_j) \in \mathcal{G} \]
for some $r^\pm_j(x)>0$.

As an example, consider the case where the domain $\Omega$ is a square, $\mathcal{M}$ is {either a hexagonal or triangular} tiling of $\R^2$, and the desired angles are $\theta_j = \frac{j\pi}{6}$, for $j = 0, \ldots, 5$.  The resulting {meshes are} pictured in Figure~\ref{fig:hexGrid}.  An example involving an underlying Cartesian grid is shown in Figure~\ref{fig:L1angles}.

\begin{figure}
    \centering
    \subfigure[]{
    \includegraphics[height=0.3\textwidth]{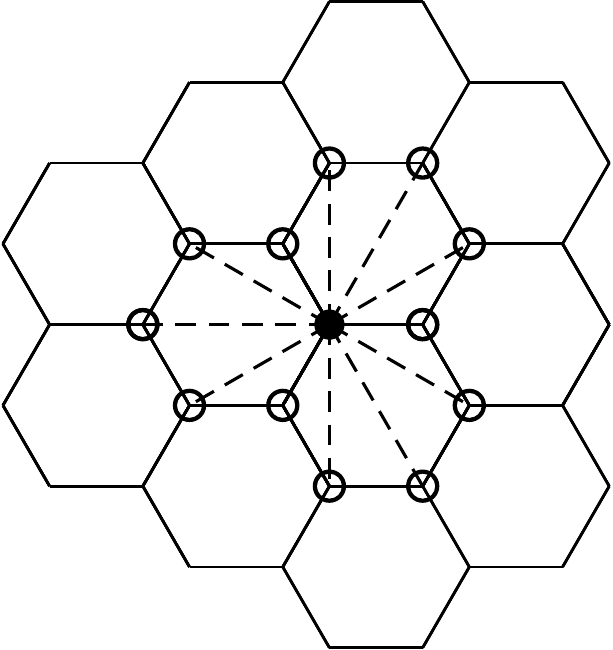}\label{fig:HexStencil}}
    \subfigure[]{\includegraphics[height=0.3\textwidth]{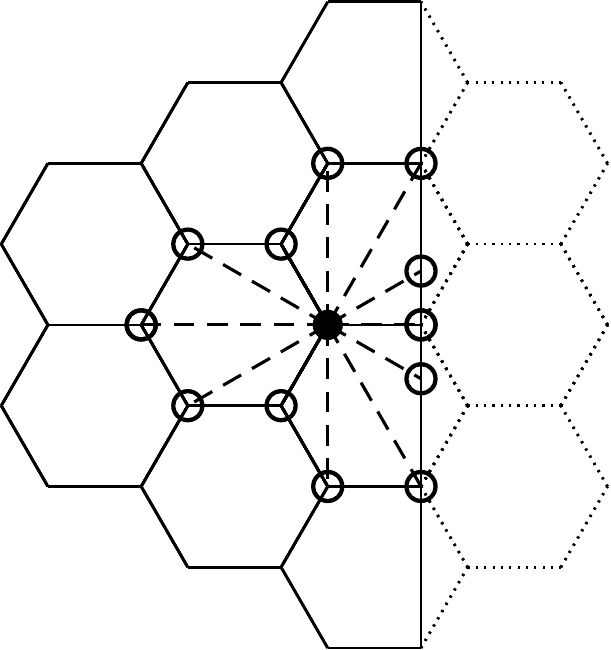}\label{fig:hexBdy}}
    \subfigure[]{\includegraphics[height=0.3\textwidth]{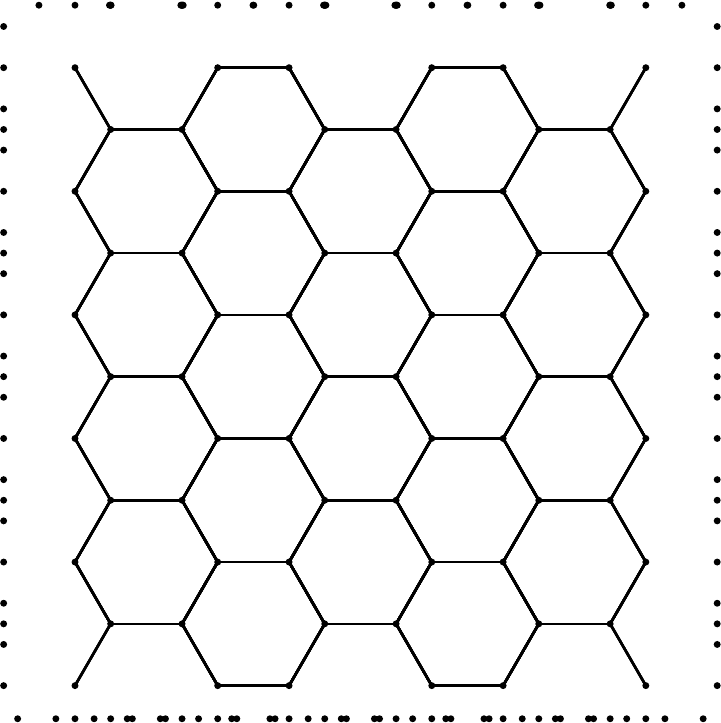}\label{fig:hexMesh}}
		    \subfigure[]{
    \includegraphics[height=0.3\textwidth]{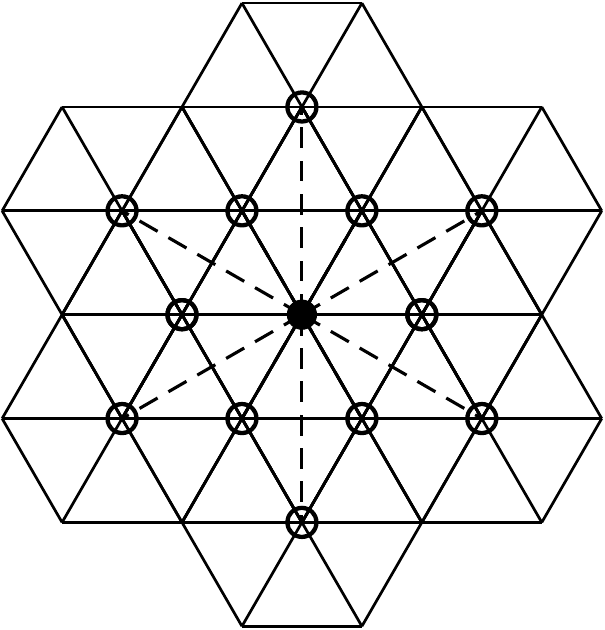}\label{fig:TriStencil}}
    \subfigure[]{\includegraphics[height=0.3\textwidth]{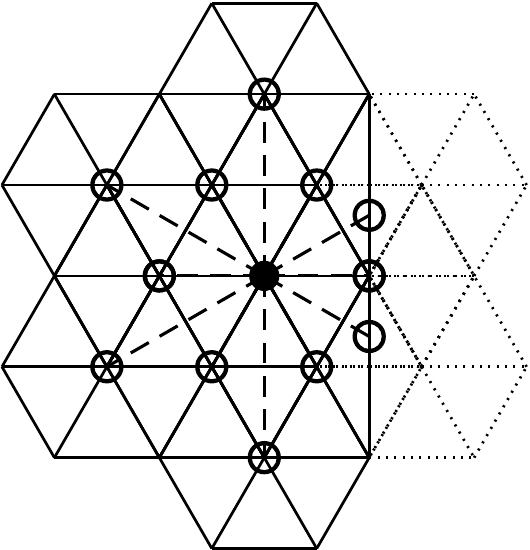}\label{fig:triBdy}}
    \subfigure[]{\includegraphics[trim={2.75cm 1.4cm 2.75cm 1.15cm},clip,height=0.3\textwidth]{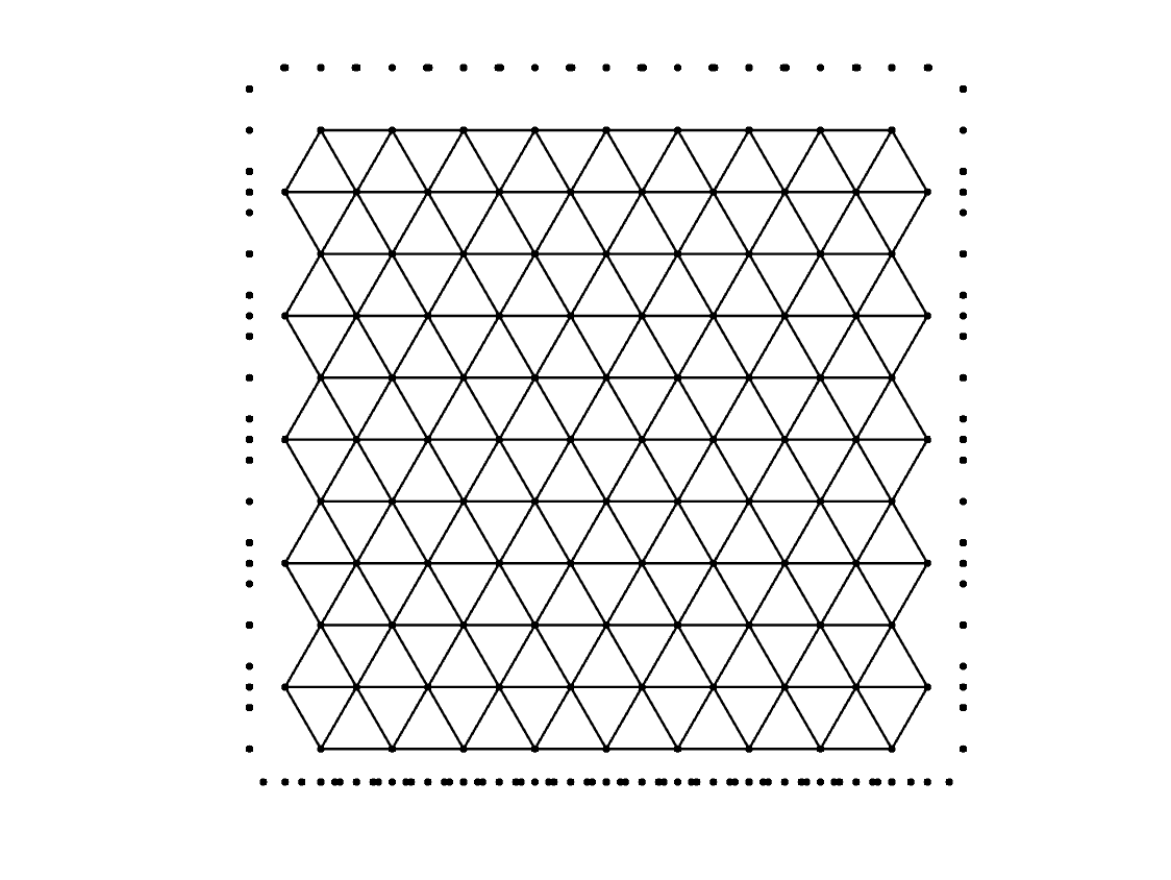}\label{fig:triMesh}}
    \caption{\subref{fig:HexStencil}~A hexagonal mesh, discrete set of angles, and neighboring mesh points aligned with those angles. \subref{fig:hexBdy}~An example of boundary points that are added to the grid. \subref{fig:hexMesh}~Meshing of a square domain using a hexagonal mesh augmented with boundary points. {\subref{fig:TriStencil}~A triangular mesh, discrete set of angles, and neighboring mesh points aligned with those angles. \subref{fig:triBdy}~An example of boundary points that are added to the grid. \subref{fig:triMesh}~Meshing of a square domain using a triangular mesh augmented with boundary points.}}
    \label{fig:hexGrid}
\end{figure}

The meshing of the domain can be easily accomplished if the domain $\Omega$ is represented through the signed distance function $d_{\partial\Omega}(x)$ to its boundary $\partial\Omega$. See Algorithm~\ref{alg:meshing}.

\begin{algorithm}[h]
\caption{Discretization of the domain $\Omega$}
\label{alg:meshing}
\begin{algorithmic}[1]
\State $\G \leftarrow \{x\in\mathcal{M} \mid d_{\partial\Omega}(x) < 0\}$
\For{$x\in\G$ such that $0 < \abs{d_{\partial\Omega}(x)}<r$}
\For{$j = 0, \ldots, M$}
\State $r^\pm \leftarrow \min\{h>0 \mid x \pm h(\cos\theta_j,\sin\theta_j) \in \mathcal{M}\}$
\If{$r^\pm(\cos\theta_j,\sin\theta_j) \notin\G$}
\State $t^\pm \leftarrow \text{Positive solution of}$ \[ d_{\partial\Omega}(x \pm t(\cos\theta_j,\sin\theta_j)) = 0\]  
\State $\G \leftarrow \G\cup \{x \pm t^\pm(\cos\theta_j,\sin\theta_j)\}$
\EndIf
\EndFor
\EndFor
\end{algorithmic}
\end{algorithm}


\subsection{Implementation on hexagonal {or triangular} grids}\label{sec:hex}
The first implementation we suggest is motivated by a desire to exploit the spectral accuracy of the trapezoid rule when applied to a uniform discretization of the angles ($d\theta_j=d\theta$ for all $j = 0, \ldots, M$).

In order to obtain as many equi-spaced angles as possible, we propose to let the underlying mesh $\mathcal{M}$ be a tiling of $\R^2$ with {either regular hexagons or equilateral triangles}.  Then we can achieve a grid-aligned scheme using the uniform angular discretization $\theta_j = \frac{j\pi}{6}$ for $j = 0, \ldots, 5$.  See Figure~\ref{fig:hexGrid}.

The resulting angular resolution is $d\theta = \frac{\pi}{6}$.  Applying the trapezoid rule~\eqref{eq:trap}, we obtain equal quadrature weights $w_i = \frac{\pi}{6}$.  

We notice that fixing $d\theta$ also has the effect of fixing the stencil width $r = \bO(h)$.  We recall that second directional derivatives in the directions $\nu = (\cos\theta, \sin\theta)$ are discretized by~\eqref{eq:sdd} as follows:
 \begin{align*}
    \Dt_{\nu\nu} u(x) &= 2\frac{h^-u(x+h^+\nu)+h^+u(x-h^-\nu)-(h^++h^-)u(x)}{h^+h^-(h^++h^-)}.
\end{align*} 
In general, {the use of narrow stencils on a hexagonal mesh leads to} stencils that are aligned but  not necessarily centered ($h^+\neq h^-$); see Figure~\ref{fig:HexStencil}. Thus the truncation error of these finite differences satisfies $\tau_{FD}(r) = h$.  {An improvement to centered stencils with $\tau_{FD}(r) = h^2$ is possible by allowing each stencil to extend across the width of two hexagons.  However, we also note that the compact stencils illustrated in Figure~\ref{fig:HexStencil} satisfy the symmetry condition discussed in \autoref{sec:error} with $p=2$.  Thus we choose to limit our implementation to the uncentered compact stencils with the expectation (which is confirmed by numerical experiments) that the resulting scheme for approximating \MA will nevertheless display second-order accuracy in the spatial resolution parameter.  The use of narrow stencils on the triangular mesh leads to stencils that are both aligned and centered so that $\tau_{FD}(r) = h^2$ automatically; see Figure~\ref{fig:TriStencil}.}

{Finally, we choose the regularization parameters $\epsilon_1>0, \epsilon_2 \geq 0$.  From the perspective of consistency error, choosing these to be as small as possible ($\epsilon_1 \ll h$, $\epsilon_2 = 0$) might seem ideal.  However, allowing $\epsilon_2 < \epsilon_1$ is undesirable for many solution methods such as Newton's method.  In the regime where $\epsilon_2 I < D^2u(x) < \epsilon_1 I$, the resulting scheme $F^h$ would be insensitive to perturbations in $u$ and the corresponding Jacobian $\nabla F^h$ would be singular.  Moreover, larger values of $\epsilon_1, \epsilon_2$ are preferable in this regime since increasing these parameters tends to improve the conditioning of the scheme and its Jacobian.  With these factors in mind, we suggest a choice of $\epsilon_1=\epsilon_2 = h^2$, which is smaller than the other terms appearing in the truncation error and will not impact the overall order of scheme.}

From Corollary~\ref{cor:truncation}, the overall formal truncation error of the quadrature scheme~\eqref{eq:quadscheme} is $\bO(h + d\theta^p)$ for every $p>0$, {though in our implementation $d\theta$ is held fixed and the scaling constant depends on $p$}.

We notice that with a fixed stencil, the truncation error $\tau_Q(d\theta)$  of the quadrature scheme does not converge to zero.  Nevertheless, at points $x\in\Omega$ where $u$ is smooth, we expect the overall truncation error of the scheme to be dominated by the remaining terms $\tau_{FD}(r)$ and $\epsilon$ unless the grid is very highly resolved.  Thus in principle the scheme~\eqref{eq:quadscheme} is not consistent.  It is certainly possible to create wider-stencil extensions of this as $h\to0$, though at the expense of a uniform angular discretization. However, in practice we expect that these wider stencils will not need to be engaged until the grid spacing $h$ is very small.
Thus we do expect to see {this scheme outperform lower-order schemes ($\bO(h^p)$, $p<2$)} for most practical refinements of the grid when solutions are smooth enough.  Indeed, this is what we observe for all but the most singular and/or degenerate of our computational examples; see~\autoref{sec:results}.   

\subsection{Implementation on Cartesian grids}\label{sec:cart}
The second implementation we propose is based upon a uniform Cartesian grid.  In order to achieve true consistency and convergence, we will allow for stencils that grow wider as the grid is refined.  To maintain grid-alignment, we are then forced to utilize a non-uniform angular discretization. This prevents the use of a spectrally accurate trapezoid rule.  However, by exploiting higher-order quadrature schemes, we can still produce monotone schemes with improved consistency error on more compact stencils.

Our particular implementation will perform quadrature using Simpson's rule, as outlined in~\eqref{eq:simp}-\eqref{eq:simpsonW}.  The truncation error of this quadrature rule is $\tau_Q(d\theta) = d\theta^4$.

Given a desired stencil width $r = Kh$ for some $K\in\mathbb{N}$, we select an angular discretization by considering neighboring grid points that are a distance $r$ from the reference point $x$ as measured by the $L^1$ norm.

Specifically, we consider a set of angles $\theta_0 < \ldots < \theta_M$ where $M=2K-1$ is odd. Letting $h$ be the standard grid point spacing in the Cartesian grid, we let $r_j$ and $\theta_j$ be the polar coordinates of the grid points
\bq\label{eq:l1angles} r_j(\cos\theta_j, \sin\theta_j) = h\left(K-j, K-\abs{K-j}\right), \quad j = 0, \ldots, 2K-1. \eq
See Figure~\ref{fig:L1angles}.

\begin{figure}
    \centering
    \subfigure[]{
    \includegraphics[width=0.35\textwidth]{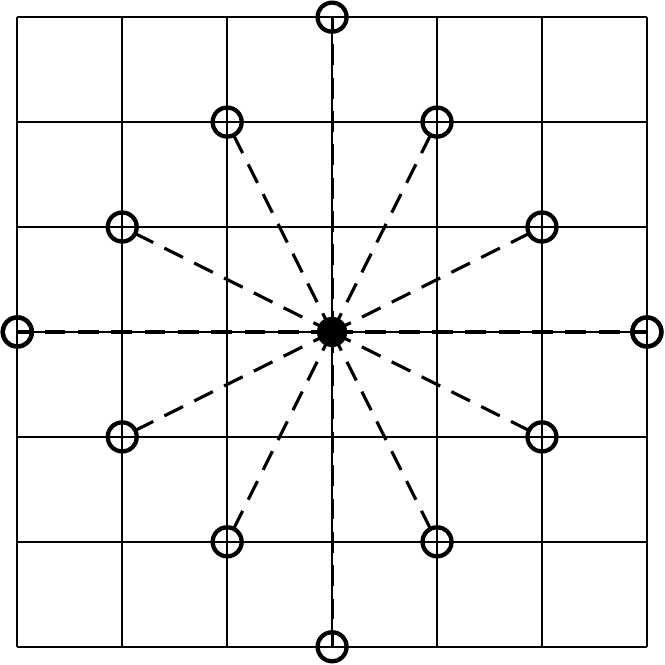}\label{fig:L1Stencil}} \hspace{20pt}
    \subfigure[]{\includegraphics[width=0.35\textwidth]{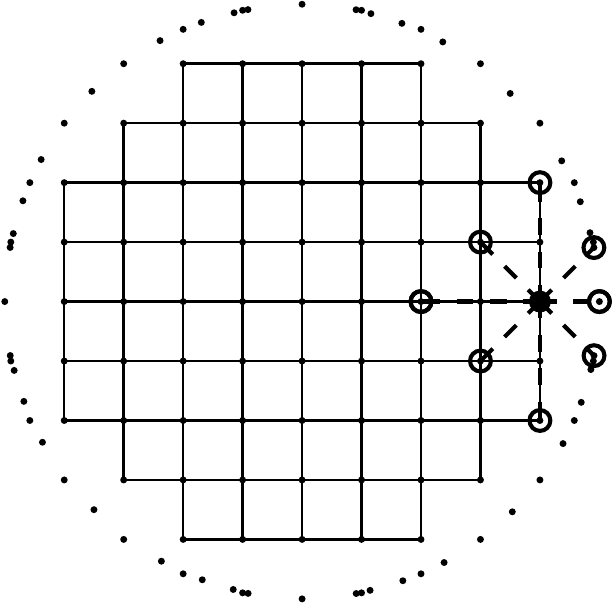}\label{fig:circdom}}
    \caption{\subref{fig:L1Stencil}~A stencil with $L^1$ width $r=3h$ on a Cartesian grid and \subref{fig:circdom}~an example set of grid points used to discretize a disc. }
    \label{fig:L1angles}
\end{figure}

An important consequence of this choice of angles is that it has a uniformly bounded quasi-uniformity constant, as required for consistency (Theorem~\ref{thm:consistency}).  Moreover, as $d\theta\to0$ the ratios
\[ \frac{d\theta_{j+1}}{d\theta_j} \to 1. \]
  This ensures that the quadrature weights~\eqref{eq:simpsonW} are strictly positive, as required for monotonicity (Theorem~\ref{thm:monotonicity}).

\begin{lemma}[Quasi-uniformity]
The angular discretization defined in~\eqref{eq:l1angles} has a uniformly bounded quasi-uniformity constant.
\end{lemma}
\begin{proof}
We bound the ratio $d\theta/d\theta_j$ for $j = 0, \ldots, K$.  The remaining cases are identical by symmetry.

Notice that the local stencil width is given by
\[ r_j^2 = h^2\left((K-j)^2+j^2\right).\]
This is bounded by
\[ \frac{hK}{\sqrt{2}}  \leq r_j \leq hK.\]

We can also compute the local angular resolution via
\begin{align*}
\sin d\theta_j &= \sin(\theta_{j+1}-\theta_j)\\
  &= \sin\theta_{j+1}\cos\theta_j - \cos\theta_{j+1}\sin\theta_j\\
  &= \frac{h^2}{r_jr_{j+1}}\left((j+1)(K-j)-(K-j-1)j\right)\\
  &= \frac{Kh^2}{r_jr_{j+1}}.
\end{align*}

Bounds on $r_j$ imply that for every $j$,
\[ \frac{\sin d\theta}{\sin d\theta_j} \leq \frac{2/K}{1/K} = 2.  \]

Thus the upper bound on the quasi-uniformity constant $Q$ converges to 2 as $K\to\infty$.
\end{proof}

\begin{lemma}[Ratios of angles]
The angular discretization defined in~\eqref{eq:l1angles} satisfies
\[ \frac{d\theta_{j}}{d\theta_{j+1}} \to 1 \]
as $K\to\infty$.
\end{lemma}
\begin{proof}
As in the proof of the previous lemma, we can use symmetry to limit ourselves to considering $j = 0, \ldots, K$ and compute 
\[ \frac{\sin d\theta_{j}}{\sin d\theta_{j+1}} = \frac{r_{j+2}}{r_{j}} = \frac{(K-j-2)^2+(j+2)^2}{(K-j)^2+j^2} = 1 + \frac{8-4K+8j}{(K-j)^2+j^2}. \]

We notice that
\[  \frac{\abs{8-4K+8j}}{(K-j)^2+j^2} \leq \frac{20K}{K^2/2},\]
which converges to zero as $N\to\infty$.

Therefore
\[ \frac{\sin d\theta_{j}}{\sin d\theta_{j+1}} \to 1\]
as $K\to\infty$.
\end{proof}

From the proofs of the previous lemmas, we notice that $d\theta = \bO(1/K)$.  Since we initially chose the search radius $r = Kh$, we find the following relationship between the grid parameters described in Definition~\ref{def:notation}:
\[ d\theta = \bO\left(\frac{h}{r}\right). \]
The uniform Cartesian grid allows us to use centered differences to discretize the second derivatives~\eqref{eq:sdd}, so that $\tau_{FD}(r) = r^2$.  We recall also that Simpson's rule satisfies $\tau_Q(d\theta) = d\theta^4$.

Combining these terms, we find that the formal truncation error of the quadrature scheme  (Corollary~\ref{cor:truncation}) is given by
\[ \bO(\tau_Q(d\theta) + \tau_{FD}(r) + \epsilon) = \bO(h^4/r^4 + r^2 + \epsilon). \]

An optimal choice is obtained by the stencil width
\[ r = \bO(h^{2/3}), \]
which leads to an angular resolution of $d\theta = \bO(h^{1/3})$.  Choosing $\epsilon \leq r^2 = \bO(h^{4/3})$, we find that the formal consistency error of the scheme is given by
\[ \bO(h^{4/3}). \]
Moreover, these choices satisfy all the requirements of consistency (Theorem~\ref{thm:consistency}).

We should remark that in a small band of radius $r$ near the boundary of $\Omega$, it may not be possible to use centered finite differences.  Instead, we must fall back on uncentered differences in~\eqref{eq:sdd} ($h^+\neq h^-$) so that $\tau_{FD}(r) = r = \bO(h^{2/3})$.  This reduces the overall truncation error of the scheme to $\bO(h^{2/3})$.  However, we emphasize that this occurs only in a narrow band, which vanishes as $h\to0$.  In our computational experiments (\autoref{sec:results}), we found that the reduced accuracy at a small number of points had \emph{no} impact on the global accuracy of the method.

We notice that with a careful selection of grid parameters, the quadrature based scheme can produce substantial improvements over monotone schemes such as the work of~\cite{FroeseMeshfreeEigs}, which requires a much larger stencil width $r =\bO(\sqrt{h})$ and produces a significantly worse truncation error $\bO(\sqrt{h})$.  This is possible because higher-order quadrature rules allow for substantial improvements in the component of the error coming from the angular resolution $d\theta$, which can be made arbitrarily small with the use of higher-order quadrature schemes.

The formal convergence of our quadrature scheme is superlinear in $h$, which is of great value when the goal is to approximate solution gradients (which is common in problems related to optimal transport).  Moreover, higher-order quadrature rules could be substituted in place of Simpson's rule to provide even greater improvements in both stencil width and truncation error.  In general, a quadrature rule satisfying $\tau_Q(d\theta) = h^p$ can be combined with a stencil width $r = h^{1-2/p}$ to produce a scheme with a formal truncation error of $\bO(h^{2-4/p})$.



\section{Computational Results}\label{sec:results}
In this section, we present numerical results for the \MA equation {(${\det}^+(D^2u(x)) = f(x)$) with Dirichlet boundary conditions ($u(x)=g(x)$).  To accomplish this, we solve a system of the form
\bq\label{eq:system}
F^h(x,u^h(x),u^h(x)-u^h(\cdot)) = \begin{cases}  G^h(x,u^h(x),u^h(x)-u^h(\cdot)) + f(x) & x \in \G\cap\Omega\\ u^h(x)-g(x), & x \in \G\cap\partial\Omega\end{cases}
\eq
where $G^h$ is a consistent, monotone approximation of the convexified \MA operator.
}

 We will compare the results of the following {four} schemes:
\begin{itemize}
    \item The quadrature scheme on a hexagonal grid described in \autoref{sec:hex}.
		\item {The quadrature scheme on a triangular grid described in \autoref{sec:hex}.}
    \item The quadrature scheme on a Cartesian grid described in \autoref{sec:cart}.
    \item The method of~\cite{HS_Quadtree}, which relies on a variational formulation of the \MA operator,
    \[ {\det}^+(D^2u) = \min\limits_{\nu_1 \cdot \nu_2 = 0} \prod\limits_{j=1}^2\max\{u_{\nu_j\nu_j},0\}, \]
    discretized using {centered differences on the Cartesian grid described in \autoref{sec:cart}}.
\end{itemize}

\subsection{Numerical Implementation} 
To discretize the domain using a Cartesian grid (for either the quadrature or variational schemes), we begin with an underlying $N\times N$ grid that contains the domain.  To discretize the domain using a hexagonal tiling, we begin with a tiling covering the domain that contains $N$ points along the vertical dimension, with (approximately) $\frac{1}{2}N$ points along the horizontal dimension at each fixed nodal value of $y$.  {To discretize the domain using a triangular tiling, we begin with a tiling covering the domain that contains $N$ points along the vertical dimension, with (approximately) $N$ points along the horizontal dimension at each fixed nodal value of $y$.}  The grids are then restricted to the interior and augmented with boundary points using Algorithm~\ref{alg:meshing}.  In {every} case, we have $N \approx \frac{k}{h}$, where the constant $k$ depends only on the size of the domain.

Each of the three discretizations we consider results in a nonlinear algebraic system of equations. We solve these using a damped Newton's method 
\begin{align*}
\nabla F^h[u_n] y_n &= -(f+F^h[u_n])\\
    u_{n+1} &= u_{n} + \alpha_n y_n
\end{align*}
where the value of $\alpha_n$ is chosen at each step to ensure that the residual $r_n = \| F^h[u_n]+f\|_\infty$ is always decreasing. We run Newton's method until the residual falls below the threshold $r_n < h^2$ {since, for more degenerate/singular example, quadratic convergence is not always observed until the residual is very small}.

To obtain an initial guess $u_0$ for the Newton solver, we first solve the following Poisson equation, which is obtained through linearization of the \MA equation~\cite{BFO_MA,FO_MATheory}:
\begin{equation}\label{eq:pinit}\begin{cases}
    {\Delta} u(x) = \sqrt{2f(x)}, & x \in \Omega
    \\
    u(x) = g(x), & x\in \pOm \nonumber.
\end{cases}\end{equation}

The solution process can be accelerated slightly by first solving the linearized problem~\eqref{eq:pinit} on a coarse $N\times N$ grid, solving the nonlinear problem via Newton's method on the same coarse grid, then interpolating onto the desired refined grid to initialize the final Newton solver.
\subsection{Representative Examples}
We test our methods using four representative benchmark examples.  For simplicity of comparison, each example is posed on a square domain.  However, we should note that this is \emph{not} a simplifying assumption for the quadrature method, which performs equally well on general convex domains.

\begin{figure}[ht]
    \centering
    \subfigure[]{
        \includegraphics[width = 0.35\textwidth]{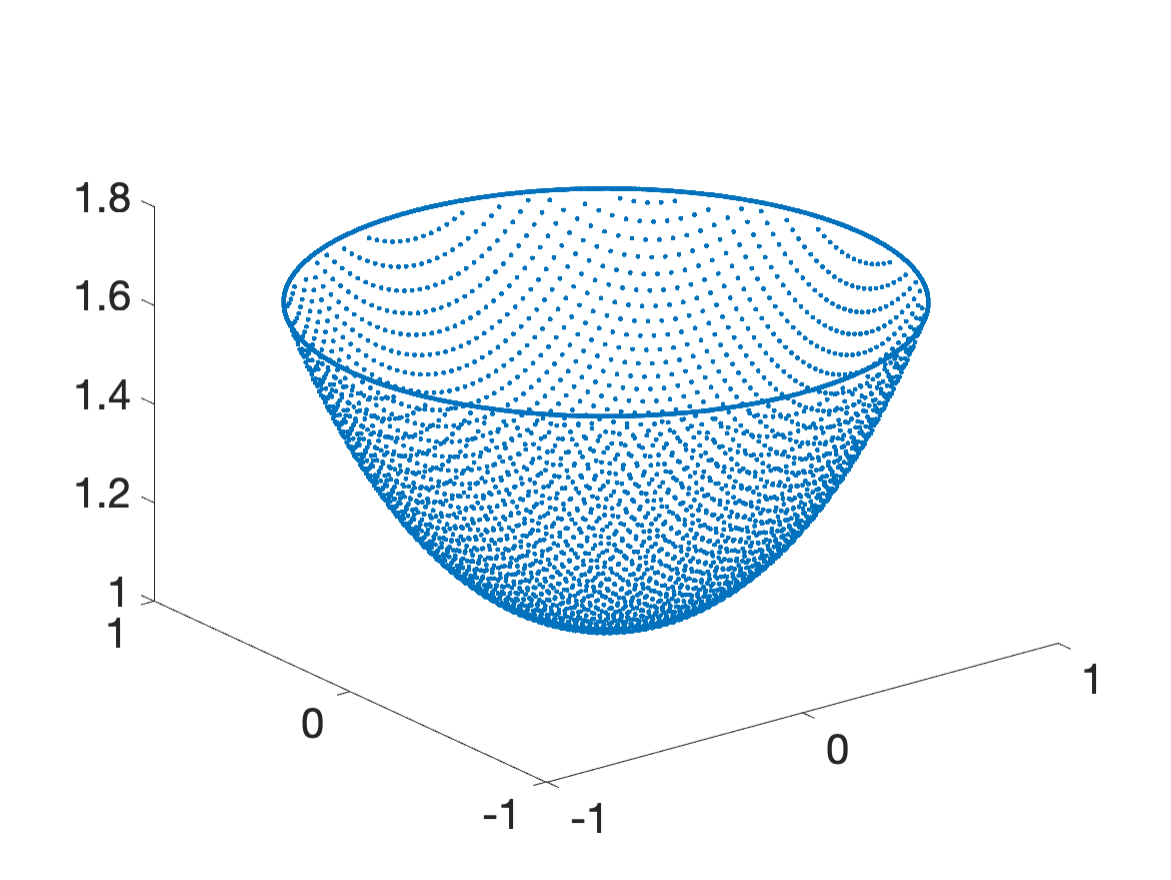}\label{fig:nsolsmth}}
    \subfigure[]{
        \includegraphics[width = 0.35\textwidth]{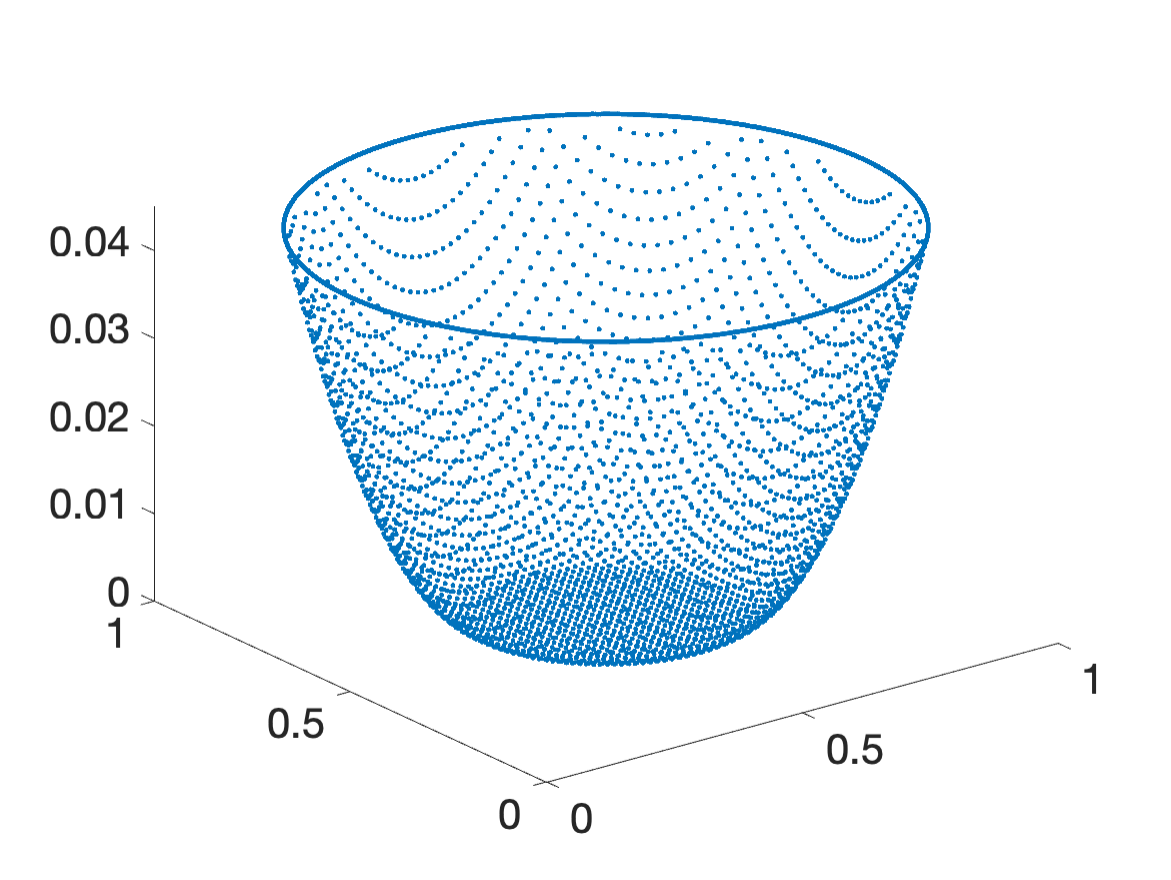}\label{fig:nsoldeg}}
    \subfigure[]{
    \includegraphics[width = 0.35\textwidth]{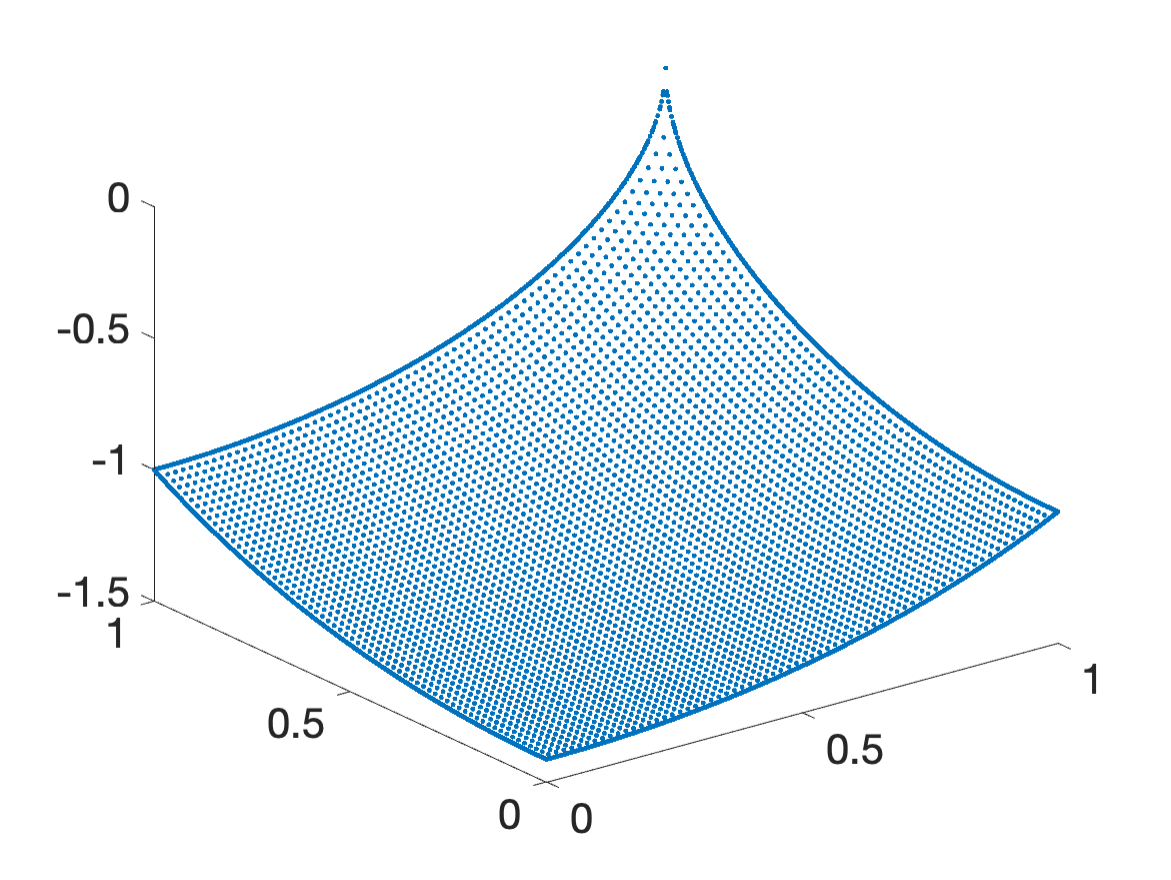}\label{fig:nsolBU}}
    \subfigure[]{
    \includegraphics[width = 0.35\textwidth]{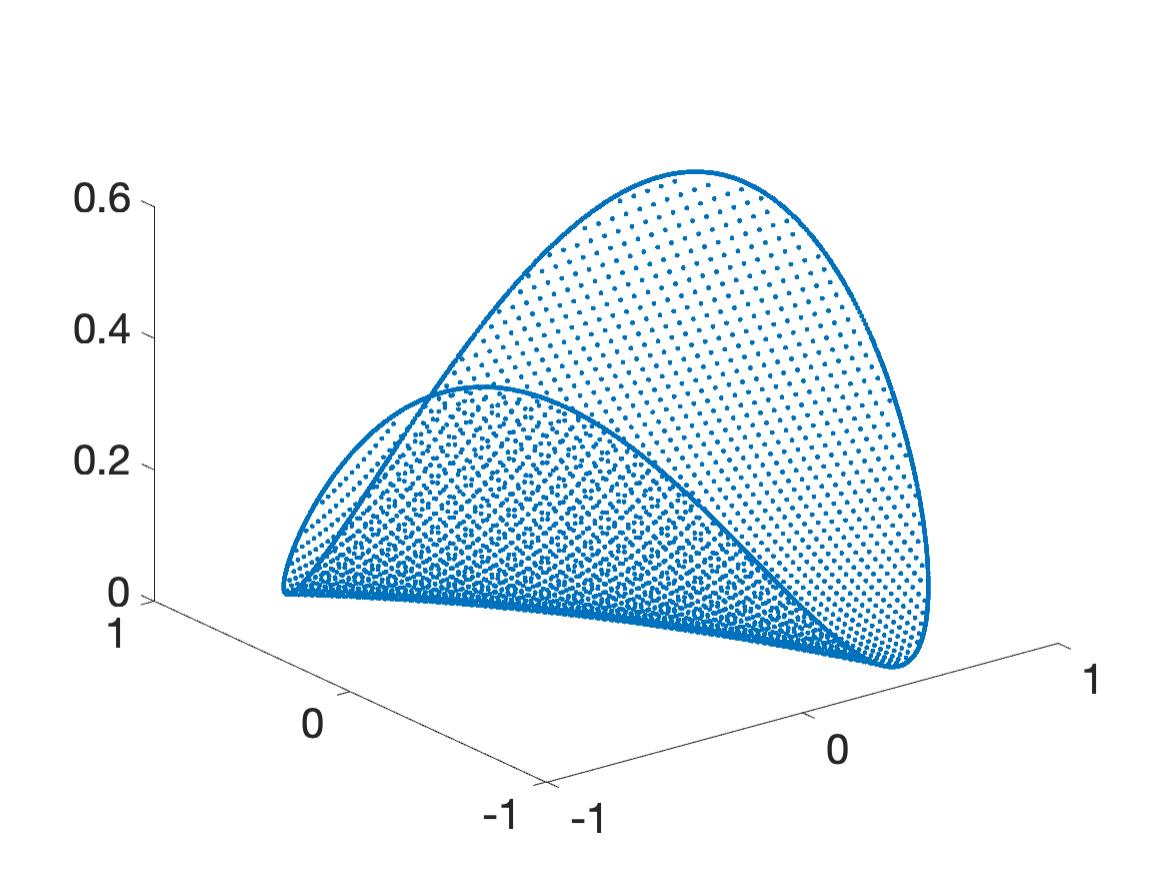}\label{fig:nsolgam}}
    \caption{Numerical solutions for \subref{fig:nsolsmth} $C^2$ example, \subref{fig:nsoldeg} $C^1$ example, \subref{fig:nsolBU} example with gradient blow-up, and \subref{fig:nsolgam} semi-degenerate example.}
    \label{fig:nsol}
\end{figure}

The first example is defined on the domain $\Omega = (-1,1)^2$ and has a smooth, radially symmetric solution $u \in C^\infty(\Omega)$.
\begin{equation}\label{eq:ex1}
    u(\textbf{x}) = \text{exp}\bigg{(}\frac{\norm{\textbf{x}}^2}{2}\bigg{)},\quad f(\textbf{x}) = \bigg{(}1+\norm{\textbf{x}}^2\bigg{)}\text{exp}\big{(}\norm{\textbf{x}}^2\big{)}.
\end{equation}

The second example is defined on the domain $\Omega = (0,1)^2$ and includes a `fully degenerate' region where both eigenvalues of $D^2u$ are 0. The solution $u \in C^1(\Omega)$ is only continuously differentiable. We introduce the constant $\textbf{x}_0 = (0.5,0.5)$ and let
\begin{equation}\label{eq:ex2}
    u(\textbf{x}) = \frac{1}{2} \big{(}(\norm{\textbf{x}-\textbf{x}_0}-0.2)^+\big{)}^2, \quad f(\textbf{x}) = \bigg{(}1-\frac{0.2}{\norm{\textbf{x}-\textbf{x}_0}}\bigg{)}^+.
\end{equation}

The third example has domain $\Omega = (0,1)^2$, and the solution is twice differentiable in the interior of the domain. However, the solution gradient becomes unbounded near the boundary point $(1,1)$.
\begin{equation}\label{eq:ex3}
    u(\textbf{x}) = - \sqrt{2 - \norm{\textbf{x}}^2},\quad f(\textbf{x}) = 2\big{(}2 - \norm{\textbf{x}}^2\big{)}^{-2}.
\end{equation}

The final example is defined on the domain $\Omega = (-1,1)^2$ and the solution $u \in C^{2}(\Omega)$ is in fact a polynomial.  We introduce the vector $\vec{\gamma} = {(\frac{1}{\sqrt{2}},1-\frac{1}{\sqrt{2}})}$ and let
\begin{equation}\label{eq:ex4}
    u(\textbf{x}) = (\vec{\gamma}\cdot \textbf{x})^2, \quad f(\textbf{x}) = 0.
\end{equation}
The solution is `semi-degenerate' on the entire domain, with $D^2u(x)$ having one positive and one vanishing eigenvalue at each point in the domain.  This fully semi-degenerate example, while somewhat artificial, should be viewed as an ``edge case'' for the quadrature scheme since the truncation error degrades in this setting (Lemma~\ref{lem:truncationSemiDeg}).

See Figure~\ref{fig:nsol} for graphs of the solutions $u$, which were obtained using the Cartesian quadrature scheme.

\subsection{Numerical Results} 
\begin{figure}[ht]
    \centering
    \subfigure[]{
        \includegraphics[width = 0.45\textwidth]{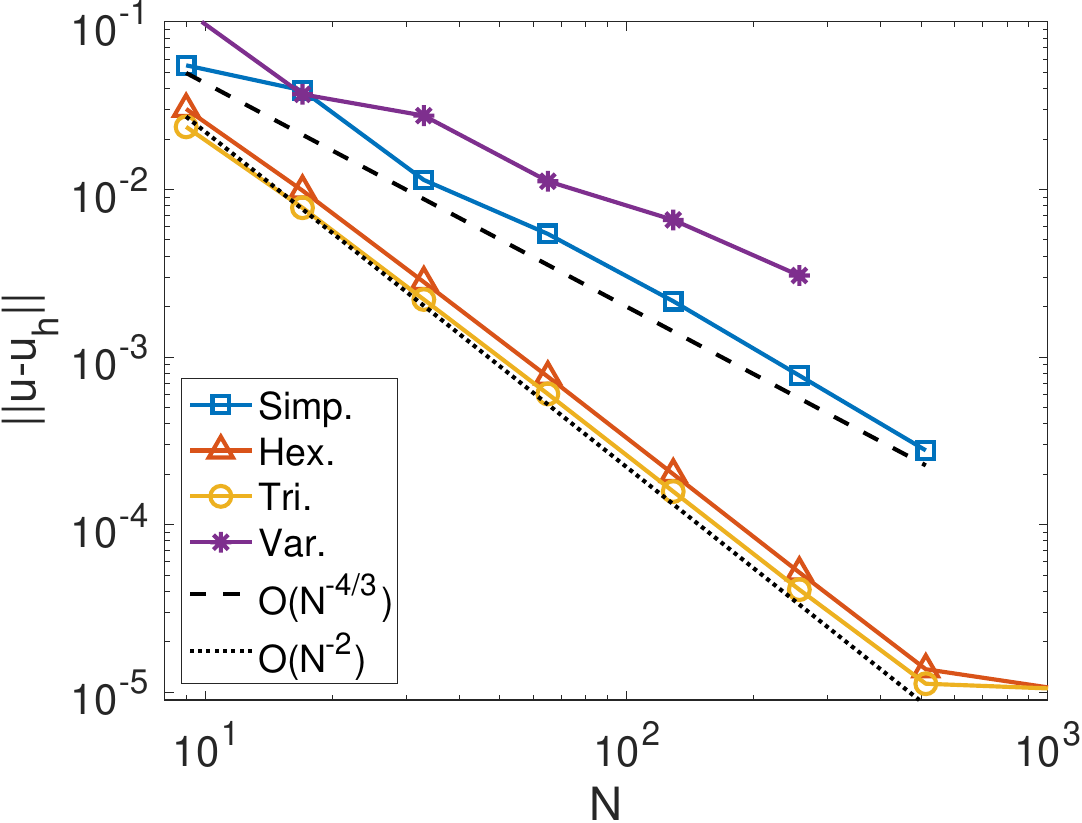}\label{fig:convsmth}}
    \subfigure[]{
        \includegraphics[width = 0.45\textwidth]{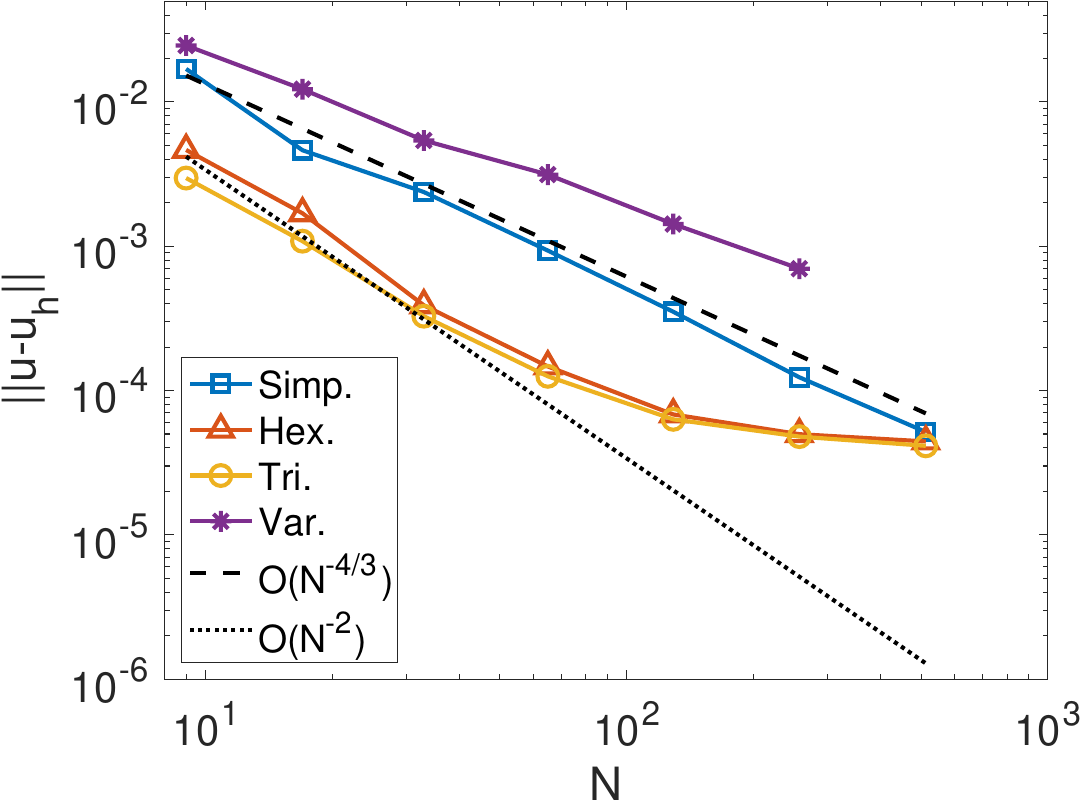}\label{fig:convdeg}}
    \subfigure[]{
    \includegraphics[width = 0.45\textwidth]{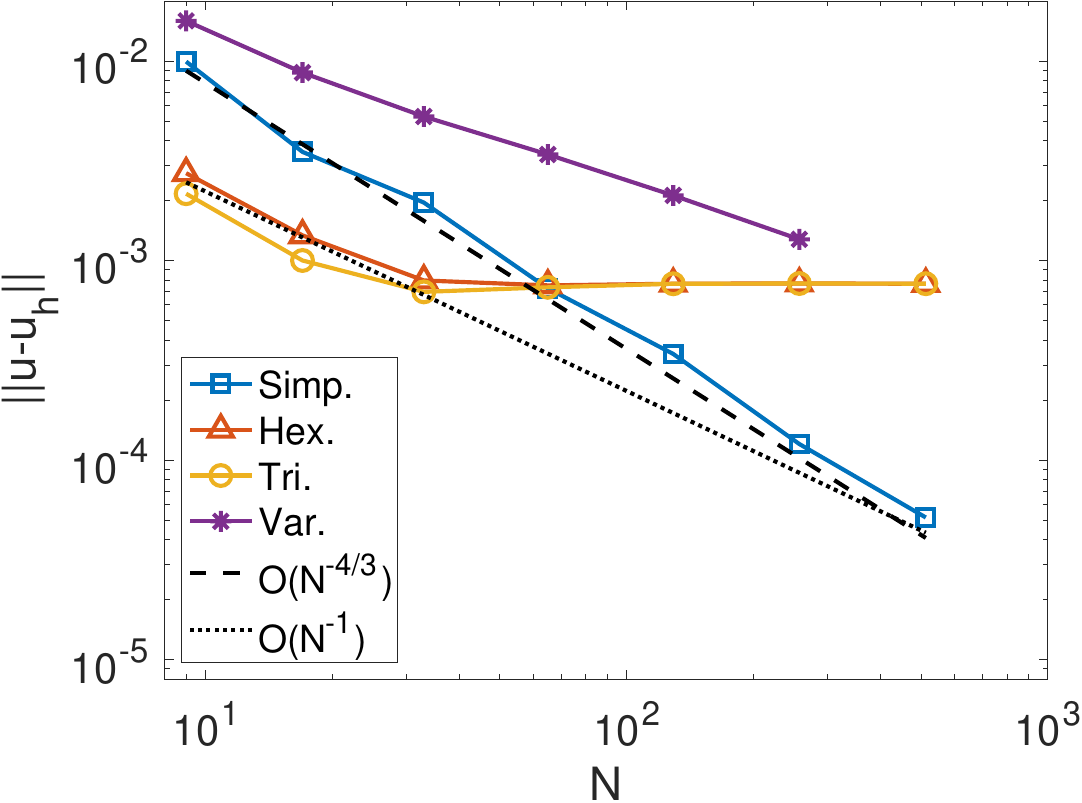}\label{fig:convBU}}
    \subfigure[]{
    \includegraphics[width = 0.45\textwidth]{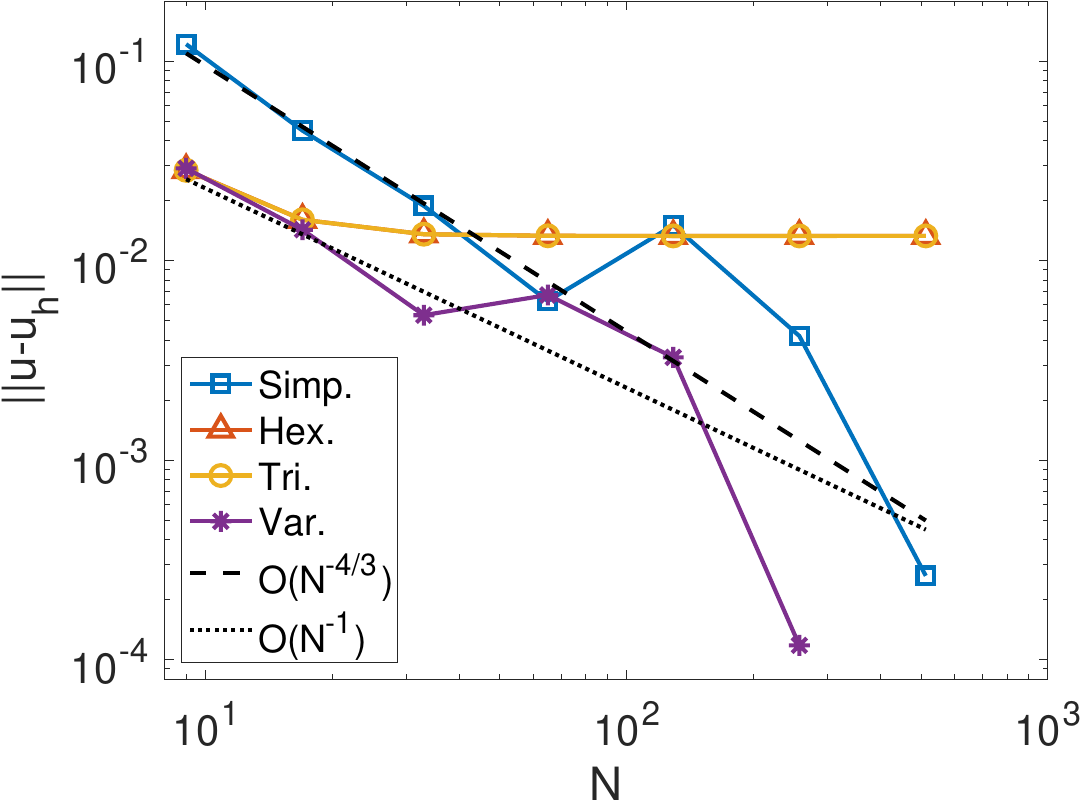}\label{fig:convgam}}
    \caption{Convergence tests for the \subref{fig:convsmth} $C^2$ example, \subref{fig:convdeg} $C^1$ example, \subref{fig:convBU} example with gradient blow-up, and \subref{fig:convgam} semi-degenerate example.}
    \label{fig:conv}
\end{figure}

The maximum error for each test is displayed in Figure~\ref{fig:conv}.  We find that the hexagonal {and triangular implementations} display effectively quadratic convergence for smooth enough tests {and small enough values of $N$}.  As expected, the error eventually levels off for less regular examples and larger values of $N$, though {these implementations} continue to outperform the others over a large range of refinements.  The Cartesian implementation of the quadrature scheme displays the expected superlinear $\bO(N^{-4/3})$ convergence; surprisingly, this continues to be true even for the less regular examples.  {On the semi-degenerate example, we observe non-monotonic convergence of the Cartesian scheme as the grid is refined.  This is likely due to the fact that the parameter $\vec{\gamma}$ is non grid-aligned; a coarser resolution can, by chance, include an angle that aligns closely to $\vec{\gamma}$, leading to improved accuracy.}

Comparison with the variational scheme demonstrates the clear superiority of the quadrature based method that are made possible by reducing the angular component of the error.  Because the variational implementation has {limited accuracy in the angular component (truncation error is $\bO(d\theta^2 + r^2$), while $d\theta$ itself goes to zero very slowly in the wide stencil schemes ($d\theta = \bO(h/r) \gg h$), solution error is at best $\bO(h)$}.

\begin{figure}[ht]
    \centering
    \subfigure[]{
        \includegraphics[width = 0.45\textwidth]{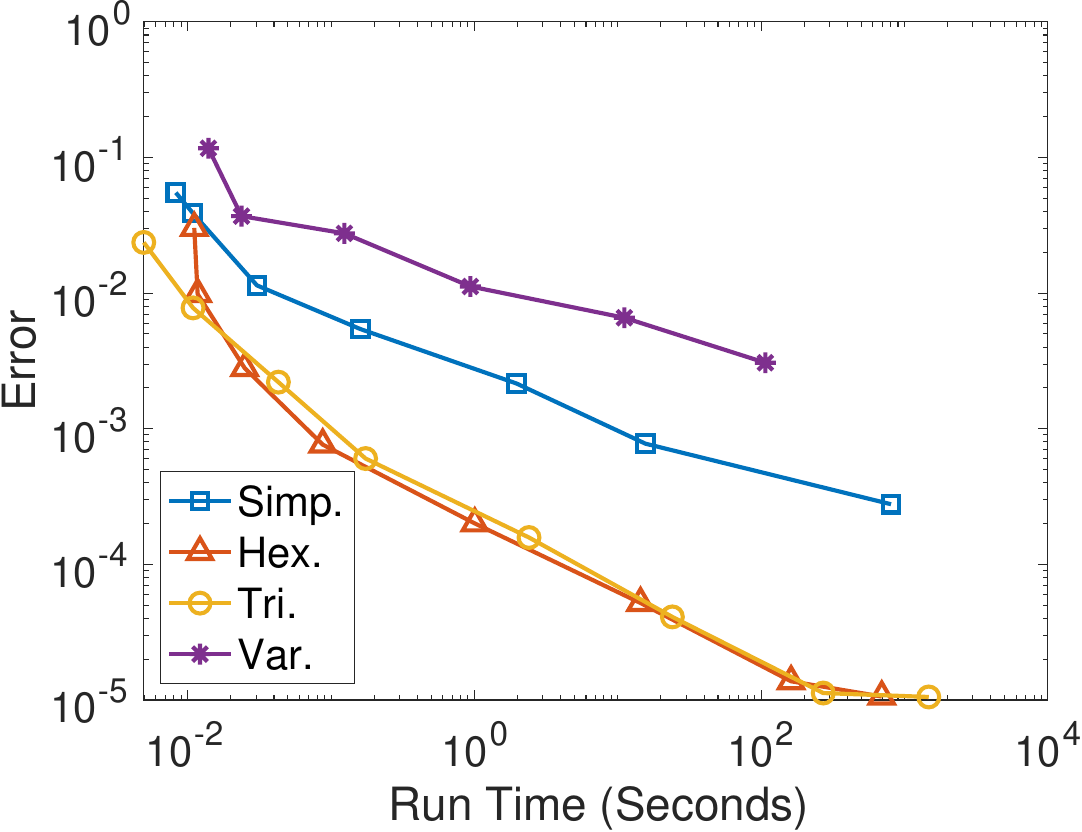}\label{fig:effsmth}}
    \subfigure[]{
        \includegraphics[width = 0.45\textwidth]{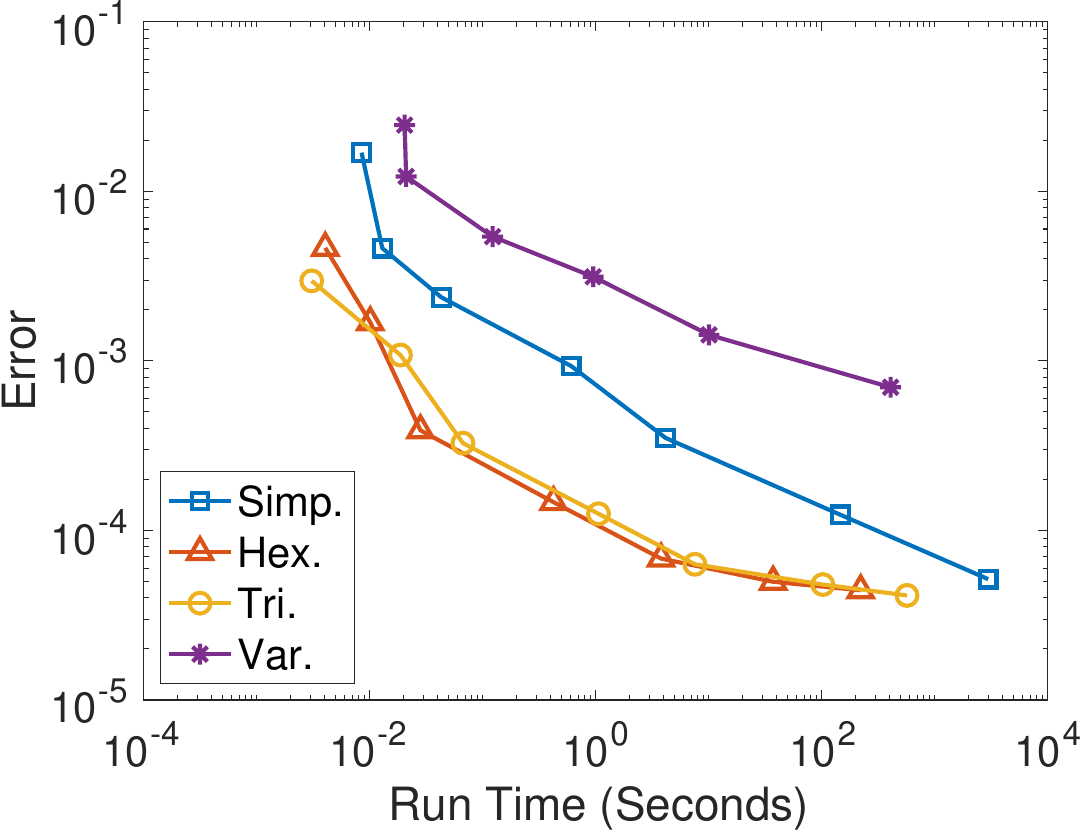}\label{fig:effdeg}}
    \subfigure[]{
    \includegraphics[width = 0.45\textwidth]{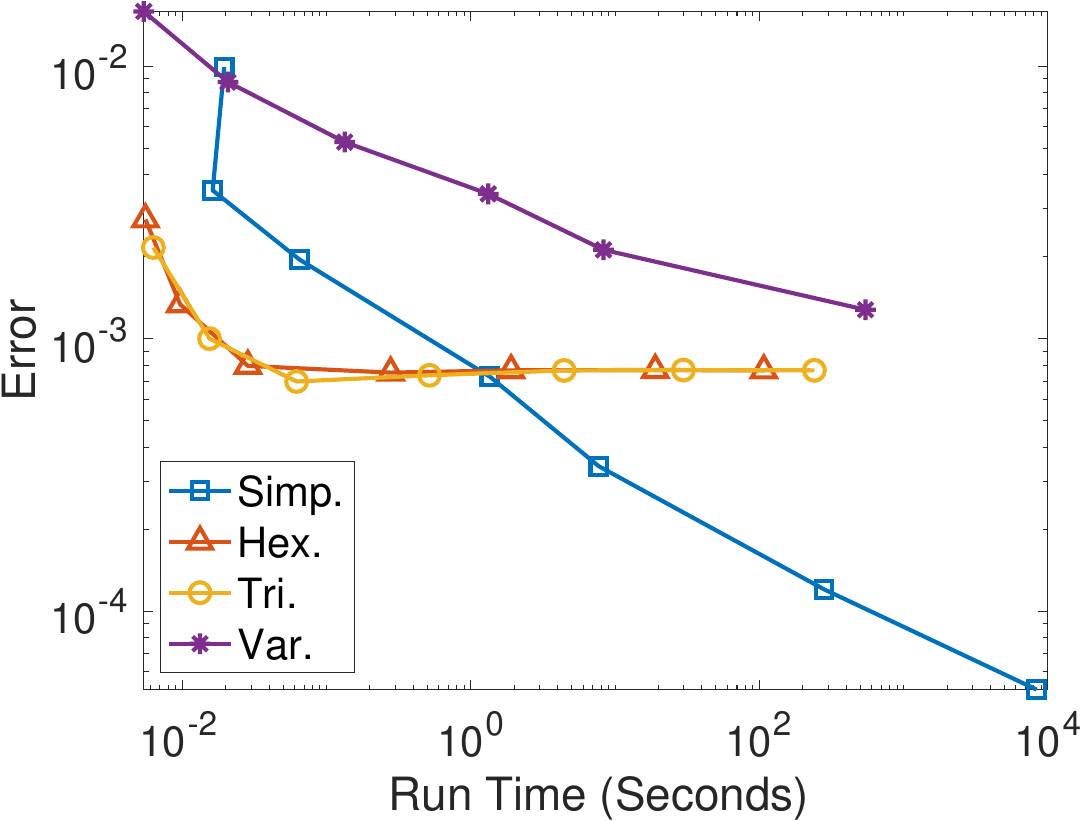}\label{fig:effBU}}
    \subfigure[]{
    \includegraphics[width = 0.45\textwidth]{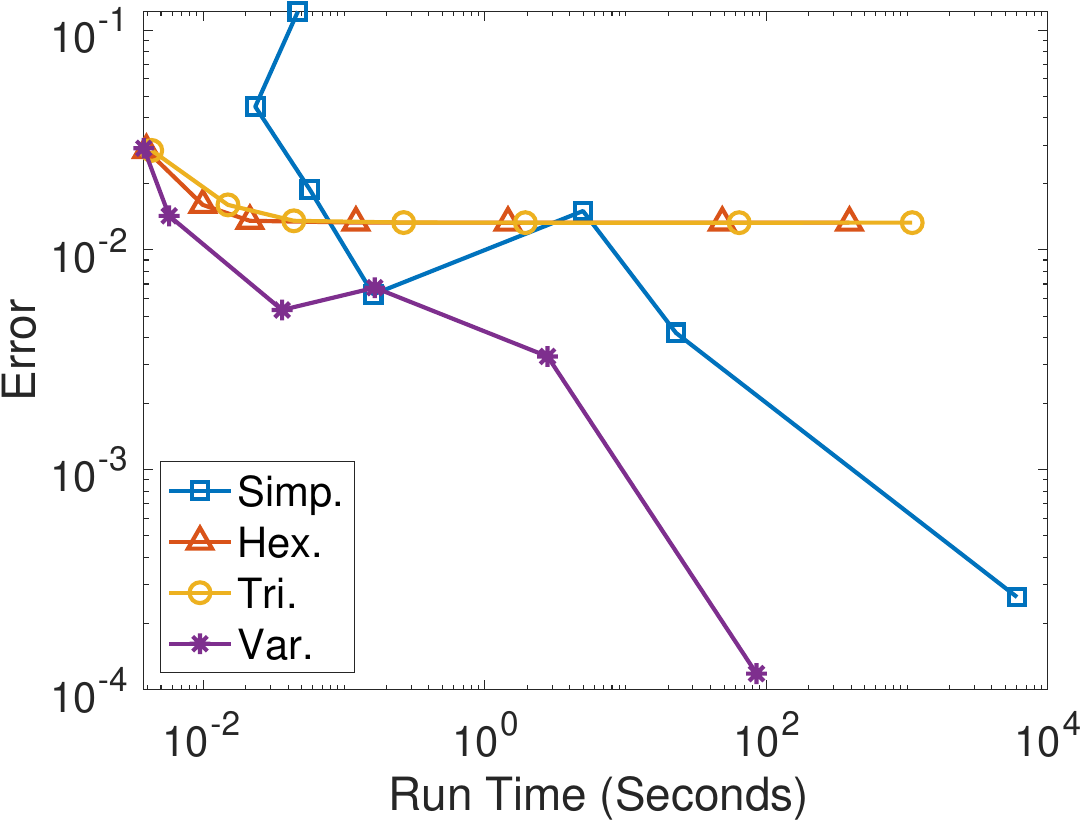}\label{fig:effgam}}
    \caption{Efficiency results for the \subref{fig:effsmth} $C^2$ example, \subref{fig:effdeg} $C^1$ example, \subref{fig:effBU} example with gradient blow-up, and \subref{fig:effgam} semi-degenerate example.}
    \label{fig:eff}
\end{figure}

The improvement achieved by the quadrature schemes becomes even more pronounced when the improvement in computational cost is factored in.  See Figure~\ref{fig:eff} for plots of solution error as a function of computation time.  It is clear that for smooth, and even moderately non-smooth examples, the hexagonal {and triangular} implementations provide the best results despite the fact that {they are} not technically consistent in the limit $N\to\infty$.  On the most singular examples (eg: blow-up in the gradient), the Cartesian implementation takes over as the most efficient.  {All} quadrature schemes dramatically outperform the variational scheme, which requires a much wider stencil ($r = \bO({\sqrt{h}})$) to optimize truncation error.  
The only exception to this trend is the semi-degenerate example.  As noted before, this can be viewed as an ``edge case'' where the variational scheme {will sometimes} perform unusually well {because (1) the centered finite difference approximations are exact on quadratics and (2) chance near-alignment between the eigenvectors of the Hessian and the underlying Cartesian grid can drastically reduce the truncation error}.  Indeed, the performance of the Cartesian quadrature scheme is still good even on this challenging test problem.

\section{Conclusion}\label{sec:conclusion}
In this paper we presented a new integral representation of the \MA operator.  We showed that this can be combined with different quadrature rules to produce a family of monotone finite difference methods.  Importantly, these methods fit directly into existing convergence proofs for the Dirichlet~\cite{FO_MATheory,Hamfeldt_Gauss,Nochetto_MAConverge,ObermanEigenvalues} or optimal transport problems~\cite{BenamouDuval_MABVP2,Bonnet_OTBC,Hamfeldt_OTBC}.

Existing monotone methods for the \MA equation rely on wide finite difference stencils.  The resulting truncation error depends upon several factors: the typical spacing of grid points $h$, the width of the stencil $r$, and the angular resolution of the stencil $d\theta$.  The use of higher-order quadrature schemes allows us to substantially reduce the component of the error coming from the angular resolution.  This, in turn, allows for significant reductions in both the stencil width $r$ and the overall truncation error of the scheme. The end result is a monotone (convergent) method that achieves significant gains in both accuracy and efficiency.

We provided {three} implementations of this method.  The first {two} combined the spectrally accurate trapezoid rule with an underlying hexagonal {or triangular} mesh.  The resulting methods involved a simple nearest-neighbors scheme which is highly efficient and achieves second-order convergence in practice for smooth enough solutions and reasonable grid refinements.  The {third} method utilized a non-uniform Simpson's rule on a Cartesian mesh.  The method is provably convergent, relies on relatively narrow stencils of width $r = \bO(h^{2/3})$, is highly robust with respect to solution regularity, and provides superlinear convergence of order $\bO(h^{4/3})$.  Moreover, this implementation could easily be adapted to accommodate other higher-order quadrature rules, making possible a formal convergence rate of $\bO(h^{2-4/p})$ for any $p>0$.

This approach holds particularly great promise for the three-dimensional \MA equation, for which existing discretizations can be prohibitively expensive~\cite{HL_ThreeDimensions}.  Typical schemes rely on some variational form of the \MA equation, which requires performing optimization over a three-dimensional set at each point in the three-dimensional domain.  A method based upon the integral reformulation would reduce this to the cost of integrating over the sphere, which is two-dimensional.  We also expect this approach to adapt well to generalized \MA equations arising in optimal transport problems in the plane~\cite{Froese_GenMA} or on the sphere~\cite{HT_OTonSphereNumerics}.

\bibliographystyle{plain}
\bibliography{MABib}
\end{document}